\documentclass[10pt]{article}

\usepackage{amsmath}
\usepackage{amsrefs}
\usepackage{amssymb}
\usepackage{amsthm}
\usepackage{array}
\usepackage{bm}
\usepackage{cases}
\usepackage{dcpic}
\usepackage{ifthen}
\usepackage{pictex}
\usepackage{accents}
\usepackage{color}
\usepackage{marginnote}
\usepackage{mathdots}
\numberwithin{equation}{section}
\newtheorem{Theorem}{Theorem}[section]
\newtheorem{Lemma}[Theorem]{Lemma}

\newtheorem{Proposition}[Theorem]{Proposition}
\newtheorem{Corollary}[Theorem]{Corollary}
\newtheorem{Definition}[Theorem]{Definition}

\newfont{\deffont}{cmbxti10}
\newfont{\german}{eufm10}
\newfont{\mymath}{cmr12}

\newcommand\lieg{\mathfrak{g}}

\newcommand\lief{\mathfrak{f}}
\newcommand\liee{\mathfrak{e}}

\newcommand\lies{\mathfrak{s}}
\newcommand{\liep}{\mathfrak p}
\newcommand{\liey}{\mathfrak y}
\newcommand{\liem}{\mathfrak m}
\newcommand{\lieh}{\mathfrak h}
\newcommand{\liex}{\mathfrak x}
\newcommand{\liel}{\mathfrak l}
\newcommand{\lieo}{\mathfrak o}
\newcommand{\liek}{\mathfrak k}

\newcommand\hook{\mathbin{\raise2.5pt\hbox{\hbox{{\vbox{\hrule height.4pt 
width6pt depth0pt}}}\vrule height3pt width.4pt depth0pt}\,}}

\newcommand\cTM{T^*\kern-2ptM}

\newcommand\sym{\mathfrak{X}}

\newcommand\thetaX{\theta_{\kern -1 pt X}}

\newcommand\CalPf{\mathcal{P}\text{\it \kern -.3pt f}}

\newcommand\TM{T\kern -2pt M}

\newcommand\mycap{\hbox{\ $\rlap{\kern -.3pt $\cap$}\raise.8pt\hbox{$\scriptstyle+$}$\ } }

\DeclareMathOperator{\rank}{rank}

\newboolean{proofmode}

\newcommand{\StTag}[1]{ \label{st:#1}
\ifthenelse{\boolean{proofmode}}{\ \marginpar{\bfquad\scriptsize st:#1} }{}      }

\newcommand{\EqTag}[1]{
\ifthenelse{\boolean{proofmode}}
{ {\label{eq:#1}}
  \stepcounter{equation} 
  \tag{\theequation \rlap{\kern 23 pt{\scriptsize eq:#1}}} 
}
{\label{eq:#1}}
 }

\newcommand{\EqRef}[1]{\eqref{eq:#1}}
\newcommand{\StRef}[1]{\ref{st:#1}}

\newcommand\CalD{\mathcal{D}}
 
\newcommand\CalI{\mathcal{I}} 
\newcommand\CalL{\mathcal{L}}

\newcommand\barM{
	\hbox{\kern 2.3 true pt 
	\vbox{\hrule width 8.5  true pt height .3 true pt \kern .9 true pt
	\hbox{\kern -2.3 true pt $M$}}}}
\newcommand\barU{
	\hbox{\kern .8 true pt 
	\vbox{\hrule width 6.5  true pt height .3 true pt \kern .9 true pt
	\hbox{\kern -.8 true pt $U$}}}}
\newcommand\barCalI{
	\hbox{\kern 4.3 true pt 
	\vbox{\hrule width 6.5  true pt height .3 true pt \kern .9 true pt
	\hbox{ \kern -4.3 true pt  $\CalI$}}}}
\newcommand\barXi{
	\hbox{\kern 1 true pt 
	\vbox{\hrule width 6.5  true pt height .3 true pt \kern .9 true pt
	\hbox{\kern 1 true pt $\Xi$}}}}
\newcommand\Largehat{\smash{\raise -7.5 pt \hbox{\rm\Large\^{}}}}   
\newcommand\LARGEhat{\smash{\raise -9.5 pt \hbox{\rm\LARGE\^{}}}}  
\newcommand\hugehat{\smash{\raise -7.5 pt \hbox{\rm\huge\^{}}}}  
\newcommand\Hugehat{\smash{\raise -7.5 pt \hbox{\rm\Huge\^{}}}}  

\newcommand\Largecheck{\smash{\raise -7.5 pt \hbox{\rm\Large\v{}}}}   
\newcommand\LARGEcheck{\smash{\raise -9 pt \hbox{\rm\LARGE\v{}}}}  
\newcommand\hugecheck{\smash{\raise -7.5 pt \hbox{\rm\huge\v{}}}}  
\newcommand\Hugecheck{\smash{\raise -7.5 pt \hbox{\rm\Huge\v{}}}}  

\newcommand\bfthetaX{{\boldsymbol{\theta_{\kern -1 pt X}}}}

\newcommand\Rtheta{{ \raise 1pt \hbox{$\scriptstyle {\boldsymbol{\theta}}$}}}
\newcommand\Ltheta{{ \lower 1pt \hbox{$\scriptstyle {\boldsymbol{\theta}}$}}}

\newcommand\Rsigma{{ \raise 1pt \hbox{$\scriptstyle \sigma$}}}

\newcommand\Reta{{ \raise 1pt \hbox{$\scriptstyle \eta$}}}

\newcommand\vecthsigma{\partial_{{\displaystyle \hat {\raise 1.3pt \hbox{$\scriptstyle \sigma$}}}^a}}
\newcommand\vectcsigma{\partial_{{\displaystyle \check {\raise 1.3pt \hbox{$\scriptstyle \sigma$}}}^\alpha}}

\newcommand\gfr{\mathfrak{g}}
\newcommand\mfr{\mathfrak{m}}

\setboolean{proofmode}{false}

\begin{document}

\title{
	Non-Rigid Parabolic Geometries of Monge Type
} 

\author{ 
Ian Anderson 
\\ Dept of Math. and Stat. 
\\ Utah State University  
\\ Logan, Utah 
\\USA 84322
\and  Zhaohu Nie 
\\Dept of Math. and Sat. 
\\ Utah State University
\\Logan, Utah 
\\USA 84322
\and Pawel Nurowski
\\ Centrum Fizyki Teoretycznej,
\\ Polska Akademia Nauk,
\\ Al. Lotnik\'ow 32/46, 02-668,
\\ Warszawa,  Poland
}

\maketitle

\begin{center}
{\bf Abstract}
\end{center}
{In this paper we study a novel class of parabolic geometries which we call parabolic geometries of Monge type. These parabolic geometries are defined by gradings such that their -1 component contains a nonzero co-dimension 1 abelian subspace whose bracket with its complement is non-degenerate. We completely classify the simple Lie algebras with such gradings in terms of elementary properties of the defining set of simple roots.  In addition we characterize  those parabolic geometries of Monge type which are non-rigid in the sense that they have nonzero harmonic curvatures in positive weights. Standard models of all non-rigid parabolic geometries of Monge type are described by under-determined ODE systems. The full  symmetry algebras for these under-determined ODE systems are explicitly calculated; surprisingly, these symmetries are all  just prolonged  point symmetries.
}

\newpage

\section{Introduction}

\newtheorem*{ThmA }{Theorem A}{}{}
\newcommand\pr{\text{pr}}
	Early in the development of the structure theory for simple Lie algebras, W. Killing  \cite{hawkins, helgason:1977a} conjectured  
	that there exists a  rank 2, 14-dimensional  simple Lie algebra $\lieg_2$ which admits a realization as a Lie algebra of vector fields on 	a 5-dimensional manifold.  
	This realization was discovered independently by F. Engel 
	and E. Cartan\footnote{Their articles appear sequentially  in 1893 in  Comptes Rendu \cite{cartan:1893a}, \cite{engel:1893a}. } and is given by the  
	infinitesimal symmetries of the rank 2 distribution in 5 variables for the  under-determined  ordinary differential equation 
\begin{equation}
	 \frac{d z}{d x}  = \biggl[  \frac{d^2 y}{d x^2} \biggr]^2.
\EqTag{StandardG}
 \end{equation}
	Recall that for any distribution  $\CalD$ defined on a manifold $M$, the  {\deffont Lie algebra of infinitesimal symmetries}  $\sym(\CalD)$  is the set of 
	vector fields $X$ on $M$ such that  $[X, \CalD] \subset \CalD$. Equation \EqRef{StandardG} subsequently re-appeared 
	as the flat model in Cartan's solution  \cite{cartan:1910a} to the equivalence problem for 
	rank 2 distributions 
		in 5 variables and in papers by Hilbert \cite{hilbert:1912a} 
	and 	Cartan  \cite{cartan:1914a} on the problem of closed form integration of under-determined ODE systems.  

	It is therefore natural to ask 
	if all simple Lie algebras admit such elegant realizations as  the infinitesimal symmetries of under-determined systems of ordinary differential equations.  We shall formulate this question within the 
	context of parabolic geometry and give a complete answer in terms of the novel concept of a {\deffont parabolic geometry of Monge type}. These geometries are defined intrinsically in terms of the $-1$ grading component and exist for all types of simple Lie algebras.
	In this paper we shall {\bf [i]} completely classify all parabolic geometries of Monge type; {\bf [ii]} identify those geometries which are non-rigid and describe the spaces of fundamental curvatures in terms of the second Lie algebra cohomology; {\bf [iii]} give under-determined ODE realizations for the standard models; and  {\bf [iv]} explicitly calculate the infinitesimal symmetries for the standard models. 
	For each classical simple Lie algebra, 
	one particular parabolic Monge geometry of depth 3 stands out from all the others. We believe that these particular geometries, listed in Theorem A, merit further study similar to that for the well-known $|1|$-gradations and contact gradations. 
		
	To explain this work in more detail, we first recall a few basic definitions from the general theory of parabolic geometry. 
	As presented in \cite{cap-slovak:2009a, yamaguchi:1993a}, the underlying structure for any parabolic geometry is a semi-simple Lie algebra $\gfr$ 
	and a vector space decomposition 
\begin{equation}
	\lieg = \lieg_k  \oplus \dots  \oplus \lieg_1 \oplus \lieg_0 \oplus \lieg_{-1} \oplus \dots \oplus \lieg_{-k} .
	\label{grading}
\end{equation}
	Such a decomposition is called a {\deffont $|k|$-grading} if: {\bf [i]} $[\lieg_i, \lieg_j ] \subset \lieg_{i+j}$;  {\bf [ii]} the negative part of this grading 
\begin{equation*} 
		\lieg_{-}  = \lieg_{-1} \oplus \dots \oplus \lieg_{-k}
\end{equation*}
	is generated by  $\lieg_{-1}$, that is,  $[\lieg_{-1}, \lieg_{\ell}] = \lieg_{-1+ \ell} $ for $\ell < 0$; 
	and  {\bf [iii]} $\lieg_k \neq {0} $ and $\lieg_{-k} \neq {0}$.  The negatively graded part  $\lieg_{-}$ is a graded nilpotent  Lie algebra  
	while the non-negative part of this grading
\begin{equation*}
	\liep  = \lieg_k  \oplus \cdots  \oplus \lieg_1 \oplus \lieg_0
\end{equation*}
	is always a parabolic subalgebra.  We remark that for a fixed choice of simple roots $\Delta^0$ of $\lieg$, there is a one-to-one correspondence between the subsets $\Sigma$ of $\Delta^0$ and the gradings of $\lieg$ \cite[pp. 292-3]{cap-slovak:2009a}. We will denote the corresponding parabolic geometry constructed this way by $(\lieg, \Sigma)$. 

	For every  $|k|$-grading of a simple Lie algebra $\lieg$, there is unique element $E\in \lieg_0$, called the {\deffont grading element}, 
	such that $[E, x] = j x$ for all $x \in \lieg_{j}$ and $-k\leq j \leq k$.  Let $\Lambda^q(\lieg_{-}, \lieg)$ be the vector space of 
	$q$-forms on $\lieg_-$ with values in $\lieg$ and set
$	\Lambda^{q}(\lieg_{-}, \lieg)_p $
	to be the subspace of $q$-forms which are homogeneous of weight $p$, that is, 
	$$
		\Lambda^{q}(\lieg_{-}, \lieg)_p= \{\omega \in \Lambda^q(\lieg_{-}, \lieg)\ |\  \CalL_E(\omega) = p\, \omega \}. 
	$$
	The spaces  $\Lambda^{*}(\lieg_{-}, \lieg)_p$ 
	define a co-chain complex with respect to the standard Lie algebra differential. The cohomology of this co-chain complex is denote by\footnote{The notation in Yamaguchi \cite{yamaguchi:1993a} is $H^{p, q}(\lieg_{-}, \lieg) = H^q(\lieg_{-}, \lieg)_{p+q-1}.$	}
	$H^{q}(\lieg_{-}, \lieg)_p$.   A parabolic geometry is called {\deffont  rigid}  if all the degree 2 cohomology spaces 
	in positive weights vanish and   {\deffont  non-rigid} otherwise. 
	The cohomology spaces $H^{q}(\lieg_{-}, \lieg)_p$ 
	can be calculated by the celebrated method of Kostant \cite{kostant} (see also \cite[\S 3.3]{cap-slovak:2009a} and \cite[\S 5.1]{yamaguchi:1993a}). 
	
	With these preliminaries dispatched, fix a $|k|$-grading of  $\gfr$, let  $N$ be the  simply connected Lie group  
	with Lie algebra $\lieg_{-}$  and let $\CalD(\lieg_{-1})$ be the  distribution on $N$ generated by the 
	left invariant vector fields corresponding to the $\lieg_{-1}$ 	component of $\lieg_{-}$. 
	This distribution is called the {\deffont standard differential system} associated to the given parabolic geometry.

	It is a fundamental result of N. Tanaka (see \cite{yamaguchi:1993a} Sections 2 and 5, especially pages 432 and 475) that if
$	H^{ 1}(\lieg_{-}, \lieg)_p  = 0$ for $p \geq 0$, 
	then the Tanaka prolongation of $\lieg_-$ coincides with $\lieg$ and 
	we have the following Lie algebra isomorphism
	\begin{equation}
	 \sym(\CalD(\lieg_{-1})) \cong \lieg. 
\EqTag{IntroSym}
\end{equation}
	{\it In this way, one can construct many examples of distributions  $\CalD$ whose symmetry
	algebra $\sym(\CalD)$  is a  given finite dimensional  simple Lie algebra $ \lieg$.} 
	Indeed, pick a subset $\Sigma \subset \Delta^0$ of the simple 
	roots and construct the associated grading \eqref{grading}, which we require to 
		satisfy $	H^{ 1}(\lieg_{-}, \lieg)_p  = 0$ for $p \geq 0$. This cohomology condition is generally satisfied, with the few exceptions enumerated in \cite[Proposition 4.3.1]{cap-slovak:2009a} or \cite[Proposition 5.1]{yamaguchi:1993a}. 
		Then calculate the left invariant vector fields on the nilpotent Lie group $N$.  
		By \EqRef{IntroSym} the Lie algebra of the infinitesimal symmetries of the standard differential system $\CalD(\lieg_{-1})$ 
	is the given simple Lie algebra  $\lieg$. Finally write down a system of ordinary or partial differential equations 
	whose canonical differential system is $\CalD(\lieg_{-1})$.  
	
	All of these calculations can be done with the Maple {\it DifferentialGeometry} package and this allowed the authors to generate many 
	examples of differential equations with prescribed simple Lie algebras of infinitesimal symmetries.  For each classical simple Lie algebra
	{\it one} particular parabolic geometry immediately stood out from all the others. These are listed in the following theorem.

\newtheorem*{ThmX}{Theorem  A}{}{}
\newtheorem*{ThmB}{Theorem  B}{}{}
\begin{ThmX}[]  
	The standard differential systems  for the  parabolic geometries   
	$A_\ell\{\,\alpha_1,\alpha_2,\alpha_3\, \}$,  
	$C_\ell\{\alpha_{\ell-1}, \alpha_\ell\}$, $B_\ell\{\alpha_1, \alpha_2\}$ and  $D_\ell\{\alpha_1, \alpha_2\}$,
	are  realized as the  canonical  differential systems for the under-determined ordinary differential equations
\begin{alignat}{2}
{\bf I}:&\  A_\ell\{\,\alpha_1,\alpha_2,\alpha_3\, \},\  \ell \geq 3,  & \dot z^i &=  \dot y^0 \dot y^i,  \ 1 \leq i \leq \ell -2.  
\EqTag{StandardA}
\\[1\jot]
{\bf II}: &\ C_\ell\{\,\alpha_{\ell-1}, \alpha_\ell\,\},\ \ell \geq 3,  & \dot z^{ij} &= \dot y^i \dot y^j,  \ 1 \leq i\leq j \leq \ell -1.
\EqTag{StandardC}
\\[1\jot]
{\bf III}: &\ B_\ell\{\,\alpha_1, \alpha_2\,\},\   \ell \geq 3,  & \dot z  &=    \frac{1}{2}\sum_{i, j=1}^{2\ell-3} \kappa_{ij}\dot y^i \dot y^j.  
\EqTag{StandardB}
\\[1\jot]
{\bf IV}: &\ D_\ell\{\,\alpha_1, \alpha_2\,\}, \   \ell \geq 4,\ D_3\{\,\alpha_1, \alpha_2, \alpha_3\,\}, & \dot z  &=    \frac{1}{2} \sum_{i, j=1}^{2\ell-4}\kappa_{ij}\dot y^i \dot y^j.      
\EqTag{StandardD}
\end{alignat}
	Here $(\kappa_{ij})$ is a symmetric, non-degenerate constant matrix of an arbitrary signature $(r, s)$, where $r+s = 2\ell -3$ for $B_\ell$ or $r+s = 2\ell -4$ for $D_\ell$. 
	The symmetry algebras of {\bf I} through {\bf IV} are isomorphic,  as real Lie algebras, 
	to $\lies\liel(\ell +1, {\mathbb R})$, $\lies\liep(\ell, {\mathbb R})$,  $\lies\lieo(r+2 , s+2)$, and  $\lies\lieo(r+2 , s+2)$,  respectively. 
\end{ThmX}

We note that the only repetition in the above list is $A_3$ and $D_3$, where the matrix $(\kappa_{ij})$ has signature $(1, 1)$, corresponding to the isomorphism $\lies\liel(4, {\mathbb R})\cong \lies\lieo(3, 3)$. 

Evidently, equations \EqRef{StandardB} and \EqRef{StandardD} are the differential equations for a curve $\gamma(x) = (x, y^i(x), z(x))$ to lie on the null cone of the metric 
$$
g = dx\, dz - \frac{1}{2} \sum_{ij} \kappa_{ij} dy^i\, dy^j.
$$
Similarly, the Monge equations \EqRef{StandardA} and \EqRef{StandardC} can be interpreted as the differential equations for curves to lie on the common null cones of families of (degenerate) quadratic forms
$$
\{dx\, dz^i - dy^0\, dy^i\} \quad \text{and} \quad \{dx\, dz^{ij} - dy^i\, dy^j\}.
$$
The geometric characterization of these families of quadratic forms and their roles as geometric structures associated to parabolic geometries are interesting problems in their own right, which we hope to address in a future publication. 

	{\it The main result of this paper is an intrinsic  characterization of those parabolic geometries arising in Theorem A, as well as the $\lieg_2$ parabolic geometries defining equation \EqRef{StandardG}}. To motivate this result, two key observations are needed. First, under-determined  systems of ordinary differential
equations such as {\bf I} -- {\bf IV} are often referred to, in the geometric differential equation literature, as {\deffont Monge equations}. As distributions these 
Monge equations are all generated by vector fields $\{X, Y_1, Y_2, \dots, Y_d\} $  such that  $[Y_i, Y_j] = 0$ and such that the $2d +1$ vector fields $\{X, Y_i,  [X,Y_i]\}$   are all point-wise independent. This first observation suggests the following fundamental definition. 

\begin{Definition} 
\StTag{MongeDef} 
A parabolic geometry 
$$
\lieg = \lieg_k  \oplus \dots  \oplus \lieg_1 \oplus \lieg_0 \oplus \lieg_{-1} \oplus \dots \oplus \lieg_{-k}
$$
is of {\deffont Monge type} if its $-1$ grading component $\lieg_{-1}$ contains a co-dimension 1 non-zero abelian subalgebra $\liey$ and $\dim \lieg_{-2} = \dim \liey$.
\end{Definition}

	The second observation is that each of the parabolic geometries arising in Theorem A, as well as the Hilbert-Cartan equation \EqRef{StandardG}, is  non-rigid. 
	These two observations
motivate our second theorem.

\begin{ThmB} 
	Let $\lieg$ be a split 
	simple Lie algebra of rank $\ell$ with simple roots $\{\alpha_1, \alpha_2, \ldots, \alpha_\ell\}$. 
	The following is a complete list of non-rigid parabolic geometries  of Monge type. 
\begin{alignat*}{3}
{\bf Ia}:&\  A_\ell\{\,\alpha_1,\alpha_2,\alpha_3\, \},\  \ell \geq 3   \quad&  {\bf Ib}:&\  A_\ell\{\, \alpha_1,\alpha_2\,\},\ \ell \geq 2
\\
{\bf IIa}:&\  C_\ell\{\,\alpha_{\ell-1},\alpha_\ell\,\},\ \ell \geq 3 \quad & {\bf IIb}:& \ C_3\{\, \alpha_1, \alpha_2,  \alpha_3\, \}
\\
{\bf IIIa}:&\ B_\ell\{\,\alpha_1,\alpha_2\,\},\ \ell \geq 2 \quad & {\bf IIIb}:&\ B_2 \{\, \alpha_2\,\}  \quad  {\bf IIIc}:\  B_3 \{\, \alpha_2, \alpha_3\,\} 
\\
\quad {\bf IIId}:&\   B_3\{\, \alpha_1, \alpha_2, \alpha_3\, \}
\\
{\bf IVa}:&\   D_\ell\{\, \alpha_1, \alpha_2\,\},\ \ell \geq 4
\\
{\bf Va}:&\   G_2\{\, \alpha_1\,\}\quad & {\bf Vb}:&\  G_2\{\, \alpha_1, \alpha_2\,\}. 
\end{alignat*}
\end{ThmB}

A number of remarks concerning Theorem B are in order. First,  
	the standard differential systems for cases {\bf Ia, IIa}, $\dots$, {\bf Va}  are precisely those given by 
	equations \EqRef{StandardA},  \EqRef{StandardC}, \EqRef{StandardB} and 	\EqRef{StandardD} (for $\kappa_{ij}$ with split signature), 
and  \EqRef{StandardG}.   
Cases {\bf Ib}, {\bf IIIb}, and {\bf IIIa} with $\ell=2$ are the only cases where $H^1(\lieg_-, \lieg)_p\neq 0$ for some $p\geq 0$. The standard models for {\bf Ib} and {\bf IIIb} are easily seen to be the jet spaces $J^1(\mathbb R^1, \mathbb R^{\ell -1})$ and $J^1(\mathbb R^1, \mathbb R^1)$. 
The standard models for {\bf IIb}, {\bf IIIc}, and  {\bf IIId} are respectively 
\begin{gather}
\dot z^1= \dot y^1 \dot y^{2}\quad
\dot z^{2}=x\dot y^{2} \quad
\dot z^3=(y^1+\dot y^1 x)\dot y^{2}\quad
\dot z^4= y^1 \dot y^1 \dot y^2,\label{exceptionalC} \\
\dot z = \ddot y^1 \dot y^2,\quad \text{and}\label{excpB1}\\
\dot z^1= \dot y^1 \dot y^2 \quad
\dot z^2= \frac{1}{2}(\dot y^2)^2 \quad
\dot z^3=  \frac{1}{2}\dot y^1 (\dot y^2)^2 \quad
\dot z^4= \frac{1}{2}\dot y^2 (x\dot y^1 \dot y^2-y^1\dot y^2-2\dot y^1 y^2).\label{excpB2}
\end{gather}
Finally, the standard differential system in  case {\bf Vb}  is simply a partial prolongation of the standard differential system for   \EqRef{StandardG} (see also \cite[\S 1.3]{yamaguchi:1993a}). We provide the details for these calculations in Section 4. 

Secondly, it is a relatively straightforward  matter to extend this classification of non-rigid parabolic Monge geometries to {\it all real simple Lie algebras}.
In the real case the $|k|$-gradings are defined by those subsets of simple roots which are disjoint from the compact roots and invariant under the Satake involution \cite[Theorem 3.2.9]{cap-slovak:2009a}. 
This requirement, our classification of parabolic Monge gradations in Theorem \StRef{ClassificationOfMonge} and the classification of real simple Lie algebras (see, for example, \cite[Table Appendix B.4]{cap-slovak:2009a}), show that, in addition to the split real forms listed in Theorem B, one only has to include the real parabolic geometries listed in Theorem A {\bf III} and {\bf IV} for $\kappa_{ij}$ of general signature. 

Thirdly, it is  rather disappointing that none of the exceptional Lie algebras $\lief_4$, $\liee_6$, $\liee_7$, $\liee_8$  appear in Theorem B but, simply stated, there are rather few non-rigid parabolic geometries 
for these algebras \cite{yamaguchi:1993a} and none of these satisfy the  Monge criteria of Theorem \StRef{ClassificationOfMonge}. (See, however, Cartan \cite{cartan:1893b} for the standard differential system for $\lief_4\{\alpha_4\}$
which is not of Monge type. In the same spirit, see \cite[p. 480]{yamaguchi:1993a} for some other linear PDE systems with simple Lie algebras of symmetries.) 

We remark also that just as the Hilbert-Cartan equation \EqRef{StandardG} arises as the reduction of 
	the parabolic Goursat equation
$$
	32u_{xy}^3 -12 u_{yy}^2u_{xy}^2 + 9 u_{xx}^2 -36u_{xx} u_{xy} u_{yy}  + 12 u_{xx} u_{yy}^3 = 0
$$
	(see \cite{cartan:1911a} and \cite{Strazzullo}), 
	one also finds in \cite[p. 414]{cartan:1911a} that the equation \EqRef{StandardB}, with $\ell =3$, 
	appears as the reduction of a certain second order system of 3 non-linear partial differential equations for 1 unknown function  in 3 independent variables. See the Ph.D. thesis of S. Sitton \cite{Spencer} for details.

And finally, with regards to the Cartan equivalence problem  associated to each of  these non-rigid parabolic geometries of Monge type, it hardly needs to be said that   the $\lieg_2$ parabolic geometry defined by $\{\alpha_1\}$ was solved in full detail by Cartan \cite{cartan:1910a}. For the
remaining interesting cases, that is, except cases {\bf Ib}, {\bf IIIb}, and {\bf IIIa} with $\ell=2$, we remark that, unlike the $\lieg_2$ equivalence problem, {\it all fundamental invariants appear in the solution to the equivalence problem as torsion.}\footnote{For the definition of the torsion of a Cartan connection, see \cite[p. 85]{cap-slovak:2009a}.}  The equivalence problems associated to  the parabolic geometries {\bf IIb}
and {\bf IIId} are quite  remarkably simple - each 
admits only a scalar torsion invariant and no curvature invariants.

	The paper is organized as follows. In  Section  2 we give a complete classification of  the grading subsets $\Sigma$ for parabolic geometries of Monge type. 
We show, 
	in particular, that for  simple Lie algebras of rank $\ell \geq 3$, there is a unique simple root $\zeta\in \Sigma$ which is connected in the Dynkin diagram to every other element of  $\Sigma$. 
	  In Section 3, we adapt the arguments of  Yamaguchi \cite{yamaguchi:1993a} to 
	describe all the  non-rigid parabolic geometries  of Monge type, thereby proving Theorem B.  
	We also describe the  cohomology spaces $H^2(\lieg_{-}, \lieg)_p$ with positive homogeneity weights (as irreducilbe 
	representations  of $\lieg_0$)  for each non-rigid parabolic geometry.  This  gives a characterization of the curvature
	for the normal Cartan connection which will play an important role in our subsequent study of the Cartan equivalence problem
	for non-rigid parabolic geometries  of Monge type.
	In Section 4,  we explicitly give the structure equations for the nilpotent 
	Lie algebras $\lieg_{-}$ for each non-rigid parabolic geometry  of Monge type.  In each case 
	we integrate these structure equations to obtain the Monge equation realizations of the standard differential systems. This  establishes Theorem A. Finally in Section 5 we use standard methods to explicitly calculate the infinitesimal symmetry generators for our standard models in Theorem A. Remarkably, these infinitesimal symmetries are all prolonged point transformations. 
	
\medskip
\noindent{\bf Acknowledgment.} The Maple software used in this research was developed with the support of Anderson's NSF grant ACI 1148331SI2-SSE. Nurowski is partially supported by the Polish National Research Center
under the grant  DEC-2013/09/B/ST1/01799.

\section{Parabolic Geometries of Monge type}
	
	In the introduction we defined the notion of a parabolic geometry of Monge type (Definition \StRef{MongeDef}) as one for which  the $\lieg_{-1}$ component
	contains  a co-dimension 1 non-zero abelian subalgebra $\liey$ satisfying $ \dim \lieg_{-2} = \dim \liey$.
	In this section we obtain a remarkable  intrinsic classification of these parabolic geometries
        in terms of the defining set of simple roots  $\Sigma$. The key to this classification is the fact that the set $\Sigma$ must contain a 
	distinguished root $\zeta$ which is adjacent to all the other roots of $\Sigma$ in the Dynkin diagram of $\lieg$ (see Theorem \StRef{ClassificationOfMonge}). 
	
	Let  $\lieg$ be a complex semi-simple Lie algebra  of rank $\ell$
	 with Cartan subalgebra $\lieh$ and roots $\Delta$, positive roots $\Delta^+$ and 
simple roots $\Delta^0$. 
	The height of a root $\beta= \displaystyle \sum_{\alpha \in \Delta^0} n_{\alpha} \alpha$ with respect to $\Sigma$ is defined as 
	$\displaystyle
	\operatorname{ht}_\Sigma(\beta)= \sum_{\alpha \in \Sigma} n_{\alpha}, 
	$
	and the set of roots with height $j$ is denoted by $\Delta_\Sigma^j$. 
	The $j$-th grading component in \eqref{grading} is
	$$
	\lieg_j = \bigoplus_{\beta\in \Delta_\Sigma^j} \lieg_\beta\ \ \text{for }j\neq 0\quad \text{and }\ \lieg_0 =  \lieh\oplus \bigoplus_{\beta\in \Delta_\Sigma^0} \lieg_\beta.
	$$
	It is clear that $\dim \lieg_{j} = \dim\lieg_{-j}$. 
			
	While we shall primarily be concerned with the case  $ \dim \lieg_{-1} > 2$, 
	we shall, nevertheless, be required to  carefully analyze the case $ \dim \lieg_{-1} =2$  since this contains the exceptional 
	Lie algebra $\lieg_2$ for the Hilbert-Cartan equation \EqRef{StandardG}. For this special case, we shall use the following.

\begin{Lemma}  
	Let $\lieg$ be a  $|k|$-graded simple Lie  algebra. If  $ \dim \lieg_{-1} =  2$,  then 
	$\rank \lieg = 2$.
\end{Lemma}
\begin{proof}
	We show that if $\rank \lieg >2$ then  $\dim \lieg_{1} >  2$.
	Let $\Sigma  \subset \Delta^0$ be any non-empty subset of the simple roots  $\Delta^0$ for $\lieg$.
	If $\rank \lieg > 2$, then 
	 the set $\Sigma$ must non-trivially intersect a set of 3 connected  simple roots $\{\alpha, \beta, \gamma\}$. 
	Then  $\alpha$,  $\alpha +\beta$, $\alpha +\beta + \gamma$,  $\beta$, $\beta +\gamma$, and 
	$\gamma$ are all roots. Regardless of which  of these 3 simple roots $\alpha$, $\beta$, $\gamma$ are in $\Sigma$, there will alway be at least 3 roots with  height 1 relative to $\Sigma$ and therefore $\dim \lieg_{-1} \geq 3$.
	For example, if the intersection  with $\Sigma$ contains just $\beta$, then  $\beta$,  $\alpha +\beta$, and $\beta +\gamma$  
	have height 1 while if the intersection contains $\alpha$ and $\gamma$, then the roots $\alpha$ and  
	$\alpha + \beta$, $\beta + \gamma$, and $\gamma$ have height 1. 
\end{proof}

\begin{Theorem} 
\StTag{Rank2Monge}
	Let $\lieg$ be a $|k|$-graded simple Lie algebra  of Monge type with $\dim \lieg_{-1} =2$. Then the 
	possiblities are:
\begin{alignat*}{3}
1.\ & A_2 \{\alpha_1, \alpha_2\}     \quad&\quad2.\ &  B_2 \{\alpha_2\}  \ (\text{the short root})     \quad &\quad3.\ &  B_2 \{\alpha_1, \alpha_2\}  \kern 100 pt\\
4.\ &  G_2 \{\alpha_1\} \ (\text{the short root})   \quad &\quad5.\ &      G_2 \{\alpha_1, \alpha_2\} 
\end{alignat*}
\end{Theorem} 
\begin{proof}
	By the above lemma  $\rank \lieg  = 2$ and hence $\lieg$ is of type 
	$A_2$, $B_2  = C_2$, or $G_2$.  The gradations not  in the above  list are $A_2 \{\alpha_1\} $, $A_2 \{\alpha_2\} $, 
	$B_2 \{\alpha_1\} $ and $G_2 \{ \alpha_2\}$, and they are not of Monge type; specifically, the gradations $A_2 \{\alpha_1\} $, $A_2 \{\alpha_2\} $, and 
	$B_2 \{\alpha_1\} $ have depth $k= 1$ while for $G_2 \{\alpha_2\}$  one easily checks 
	that $\dim \lieg_{-1} = 4$ and $\dim \lieg_{-2} = 1$.
\end{proof}

	For the rest of this section we focus on the case 	$\dim \lieg_{-1} >  2$. 
	
\begin{Proposition}  
\StTag{leader}
	Let $\lieg$ be a $|k|$-graded semi-simple Lie algebra  of Monge type with $\dim \lieg_{-1} > 2$, and let $\Sigma$ be the subset of simple roots which defines the gradation of $\lieg$. 

\noindent
	{\bf [i]} The abelian subalgebra  $\liey \subset  \lieg_{-1}$ is  $\lieg_0$-invariant.

\noindent
	{\bf [ii]} There is a 1-dimensional $\lieg_0$-invariant subspace  $\mathfrak{x}$ such that   $\lieg_{-1} = \mathfrak{x} \oplus \liey$.  

\noindent
	{\bf [iii]} There is a unique 
	simple root $\zeta \in \Sigma$  and roots $\{ \beta_1, \beta_2\ , \ldots,  \beta_d\} \subset \Delta_\Sigma^1 $ such that 
\begin{equation}
	\mathfrak{x} = \lieg_{-\zeta} \quad\text{and} \quad  \liey = \lieg_{-\beta_1} \oplus \lieg_{-\beta_2 }\oplus\dots \oplus  \lieg_{-\beta_d}.
\EqTag{lieydecomp1}
\end{equation}

\noindent
	{\bf [iv]}
	The set $\Sigma$ consists precisely of the root $\zeta$ and all roots 	
	adjacent to $\zeta$ in the Dynkin diagram for $\lieg$. 
	
\noindent
	{\bf [v]}
	If $\lieg_0$ contains no simple ideal of $\lieg$, then the Lie algebra $\lieg$ is simple.
\end{Proposition}

\begin{proof}
	{\bf [i]}
	Let  $\{y_1,y_2, \dots, y_d\}$ be a basis for $\liey$ and let $\{x, y_1,y_2, \dots, y_d\}$ be a basis for $\lieg_{-1}$.
	The generating condition $[\lieg_{-1}, \lieg_{-1}] = \lieg_{-2}$ and the fact that  $\dim \lieg_{-2} = \dim \liey$ 
	imply that 
	\begin{equation}
	\label{adx-isom}
	\operatorname{ad}_x : \liey \to \lieg_{-2} \text{ is an isomorphism}.
	\end{equation}
	This implies that the vectors $z_i = [x, y_i]$ 
	form a basis for $\lieg_{-2}$. Let $u \in  \lieg_0$. Since the action of $\lieg_0$  on $\lieg$ preserves the $|k|$-grading, it follows that
\begin{equation*}
	[u, y_i] =  a_i x + b_i^j y_j .
\end{equation*}
	Since the vectors $y_i$ commute, the Jacobi identity for the vectors $u$, $y_i$, $y_j$  yields
\begin{equation*}
	a_i z_j  - a_jz_i = 0 \quad\text{for all }  1 \leq i < j \leq d.
\end{equation*}
	Since $d >1$ this implies that $a_i= 0$ and hence $[u, y_i] \in \liey$. This proves {\bf [i]}. 

\smallskip
\noindent	
{\bf [ii]} Since $\lieg$ is a complex semi-simple Lie algebra, 
	$\lieg_0$ is a reductive Lie algebra and the center ${\mathfrak z}(\lieg_0)\subset \lieh$ by \cite[Theorem 3.2.1]{cap-slovak:2009a}. Hence the center acts on $\lieg_{-1}$ by semi-simple endomorphisms. Therefore the representation of $\lieg_0$ on $\lieg_{-1}$ is completely reducible (see for example \cite[p. 316]{cap-slovak:2009a}). 
	Thus the $\lieg_0$-invariant subspace $\liey$ admits a $\lieg_0$-invariant complement $\liex$.

\smallskip
\noindent	
{\bf [iii]}
	Since the Cartan subalgebra $\lieh$ of  $\lieg$  used to define the root space decomposition of $\lieg$ is, by definition, contained in  $\lieg_0$,  
	the $\lieg_0$-invariant subspaces  $\liex$ and $\liey$ must be direct sums of the (1-dimensional) root spaces corresponding to
	roots in  $\Delta_\Sigma^{1}$.  This proves equation \EqRef{lieydecomp1}.

	 Put $\liex =  \lieg_{-\zeta} $.  In order to complete the proof of {\bf[iii]}, we must verify  that $\zeta$ is a simple root.  
	Suppose not.  Since $\zeta$ is  a positive root of height 1, we can therefore write $\zeta  = \zeta' + \zeta''$,
	where $\zeta'$ is a positive root  of height 0 and  $\zeta''$  is a positive root of height 1. Then, on the one hand, 
\begin{equation*}
	[\lieg_{\zeta'}, \liex] =  [\lieg_{\zeta'}, \lieg_{-\zeta}] = \lieg_{- \zeta''}.
\end{equation*} 
	On the other hand, $\lieg_{\zeta'} \subset \lieg_0$ and so $[\lieg_{\zeta'}, \liex]  \subset \liex$ since $\liex$ is $\lieg_0$-invariant.  This contradicts
	the above equation and therefore $\zeta$ must be a simple root which belongs to $\Sigma$.
 
\smallskip
\noindent	
{\bf [iv]}
	Let $\beta \in \Sigma\backslash {\zeta}$ and
	let $x  \in \lieg_{-\zeta}$ and $y \in \lieg_{-\beta}$  be non-zero vectors. By \eqref{adx-isom}, 
	$[x,y] \in \lieg_{-2}$ is non-zero,
	 $\zeta + \beta$ must be a root, and therefore $\beta$ is adjacent to $\zeta$ in the Dynkin diagram for $\lieg$.   
	Conversely, let $\beta$ be any simple  root adjacent to $\zeta$.  Then $\beta + \zeta$ is a  root  
	and $[\lieg_{-\beta}, \lieg_{-\zeta}] = \lieg_{-\beta - \zeta}$.
	If  $\beta\notin \Sigma$, then $\beta \in\Delta^0_\Sigma$ and therefore,  by the $\lieg_0$-invariance of $\lieg_{-\zeta}$, 	
	$[\lieg_{-\beta}, \lieg_{-\zeta}] \subset \lieg_{-\zeta}$.
	This is a contradiction and hence $\beta\in \Sigma$.
	
\smallskip
\noindent	
{\bf [v]}
	Suppose that $\lieg = \liel \oplus \liek$, where  $\liel $ and $\liek$ are semi-simple. The condition that $\lieg_0$ contains no simple ideal of $\lieg$ 
	implies that $\Sigma$ must contain simple roots of $\liel$ and $\liek$. Therefore $\Sigma$ is disconnected in the Dynkin diagram of $\lieg$, which contradicts {\bf [iv]}. 
\end{proof}
	
	In view of {\bf [v]}, we henceforth assume that $(\lieg, \Sigma)$ is a parabolic geometry of Monge type with $\lieg$ simple. By Proposition \StRef{leader}, there is a simple root $\zeta$ such that all the other roots in $\Sigma$ are connected to $\zeta$ 
	in the Dynkin diagram for $\lieg$.  We say that the root $\zeta$ is the {\deffont leader} of $\Sigma$. However not every simple root of $\lieg$ can serve as a leader for a parabolic geometry of Monge type. 	 To complete our characterization, we now turn our attention to the gradation of $\lieg$ by the leader $\zeta$ itself, and in particular to the decomposition of the semi-simple part $\lieg_{\zeta, 0}^{ss}$ of its $0$-grading component. 
	By virtue of the connectivity of $\Sigma$, there is a  one-to-one correspondence between the remaining roots 
	$\Sigma\backslash\zeta$ and the connected components of 
	graph obtained by removing the node $\zeta$ in the Dynkin diagram for $\lieg$.  
	Label these  connected components 
	by $\Upsilon_\alpha$ for  $\alpha \in \Sigma\backslash\zeta$ so that
\begin{equation*}
	\Delta^0 = \{\zeta\}\cup \bigcup_{\alpha\in \Sigma \backslash \zeta} \Upsilon_\alpha.
\end{equation*}
	Let  $\lieg(\Upsilon_\alpha)$ be the complex simple Lie algebra 
	with Dynkin diagram
	$\Upsilon_\alpha$.
		Then by \cite[Proposition 3.2.2]{cap-slovak:2009a} we have the following decomposition 
	$$
	\lieg_{\zeta, 0}^{ss} = \bigoplus_{\alpha\in \Sigma\backslash \zeta} \lieg(\Upsilon_\alpha).
	$$

\begin{Theorem}
\StTag{ClassificationOfMonge}
	Let $\lieg$ be a parabolic geometry of a complex simple Lie algebra as determined by the set of simple 
	roots  $\Sigma$.  If  $\dim \lieg_{-1} > 2$, then  $\lieg$ is a parabolic geometry of Monge type if and only if

\noindent
	{\bf [i]} there is root $\zeta \in \Sigma$ which is adjacent to every other root in $\Sigma$ in the Dynkin diagram of $\lieg$; and

\noindent
	{\bf [ii]}  For each $\alpha\in \Sigma\backslash \zeta$, the parabolic geometry for the complex simple Lie algebra $\lieg(\Upsilon_\alpha)$ defined by the root $\{\alpha\}$ is 
	 $|1|$-graded.
\end{Theorem}

	In order to prove this theorem, we consider the set of roots $\Upsilon^1_\alpha$ of $\lieg(\Upsilon_\alpha)$ with height 1 relative to the gradation by $\{\alpha\}$, that is, 
\begin{equation} 
\label{def-upsilon1}
	\Upsilon^1_\alpha   = \{\beta \in \Delta\  |\  \beta  = \alpha +   \sum_{i=1}^m n_i \beta_i  \quad 
	\text{where $\beta_i  \in \Upsilon_\alpha \backslash \alpha$, $n_i > 0$, and $m\geq 0$} \}.
\end{equation}
	Furthermore, define  subspaces of $\lieg$ by
\begin{equation}
	\liey_{-\alpha}  =   \bigoplus_{\beta\in\Upsilon^1_\alpha} \lieg_{-\beta}.
\EqTag{lieyalpha}
\end{equation}
These are the $-1$-grading components of $\lieg(\Upsilon_\alpha)$ with respect to $\{\alpha\}$. 
	The proof of Theorem \StRef{ClassificationOfMonge} depends on the following lemma. 

\begin{Lemma}   \StTag{rootfacts}
Let $\Sigma$ be a set of simple roots satisfying condition {\bf [i]} of Theorem \StRef{ClassificationOfMonge}. 

\smallskip
\noindent
{\bf [i]} Then 
$
\Delta^1_\Sigma = \{\zeta\} \cup \bigcup_{\alpha\in \Sigma\backslash\zeta}\ \Upsilon^1_\alpha,
$
and hence we have the following decomposition
\begin{equation}
	\lieg_{-1}  =  \lieg_{-\zeta} \oplus \bigoplus_{\alpha \in \Sigma\backslash\zeta}\liey_{-\alpha}. 
\EqTag{lieydecomp}
\end{equation}

\smallskip
\noindent
{\bf [ii]} If  $\beta \in \Upsilon^1_\alpha$ and $\beta' \in \Upsilon^1_{\alpha'}$ with $\alpha \neq \alpha'$, then $\beta +\beta'$ is not a root, and hence
\begin{equation}
	[\liey_{-\alpha}, \liey_{-\alpha'}]  = 0.
\EqTag{bracketliey}
\end{equation}

\smallskip
\noindent
{\bf [iii]}  If $\beta \in  \Upsilon^1_\alpha$ then $\zeta +\beta \in  \Delta$, and hence $\dim [\lieg_{-\zeta}, \liey_{-\alpha}] = \dim \liey_{-\alpha}$. 

\smallskip
\noindent
{\bf [iv]} If $\gamma \in \Delta^0_\Sigma$,  $\beta \in  \Upsilon^1_\alpha$  and $\gamma + \beta \in \Delta$, then $\gamma +\beta \in  \Upsilon^1_\alpha$. Thus the $\liey_{-\alpha}$ in \EqRef{lieydecomp} are $\lieg_0$-invariant subspaces of $\lieg_{-1}$.
\end{Lemma}

\begin{proof} {\bf [i]} 
	Clearly $\Upsilon^1_\alpha \subset \Delta_\Sigma^1$ and so it suffices to show that if 
	$\beta \in \Delta^1_\Sigma\backslash \zeta$ then there is  a root $\alpha \in \Sigma$ such that $\beta \in \Upsilon^1_\alpha$.  
	Indeed, since $\beta$ has  height 1 with respect to $\Sigma$, there is a root $\alpha \in \Sigma$  
	and simple roots $\beta_i \in \Delta^0\backslash\Sigma$ such that 
\begin{equation}
	\beta  = \alpha +   \sum_{i=1}^m n_i \beta_i  \quad 
	\text{where $n_i > 0$ and $m\geq 0$}.
\end{equation} 
	Since $\beta$ is a root, the set of simple roots   $ \{\alpha,\beta_1, \dots,  \beta_m\}$  must define  a connected subgraph of the Dynkin diagram for $\lieg$. 
	Therefore $\{\alpha,\beta_1, \dots,  \beta_m\} \subset \Upsilon_\alpha$. 
	This equation implies that  $\beta_i \in \Upsilon_\alpha\backslash \alpha$ and 
	$\beta \in  \Upsilon^1_\alpha$. 

\smallskip
\noindent
{\bf [ii]}
	In view of \eqref{def-upsilon1}, the roots $\beta \in \Upsilon^1_\alpha$ and $\beta' \in \Upsilon^1_{\alpha'}$, with $\alpha \neq \alpha'$, are given by  
\begin{equation}
	\beta  = \alpha +   \sum_{i=1}^m n_i \beta_i   \quad\text{and}\quad \beta'  = \alpha' +   \sum_{i=1}^{m'} n_i' \beta'_i. 
\end{equation}
	Since $\Upsilon_\alpha$ and $\Upsilon_{\alpha'}$ are disjoint, 
	the totality of roots $\{\alpha,  \alpha' ,  \beta_i   ,\beta'_i \}$ is not a connected subgraph  in the Dynkin diagram of $\lieg$ and therefore 
	$\beta + \beta'$ can not be a root. Consequently $[\lieg_{-\beta}, \lieg_{-\beta'}]=0$ and \EqRef{bracketliey} follows. 
	
\smallskip
\noindent
{\bf [iii]} Let $(\cdot, \cdot)$ be the positive-definite inner product on the root space induced from the Killing form.
	Since $\zeta$ is adjacent to $\alpha$ but not any of the $\beta_i$, it follows that
	$$(\beta, \zeta)  = \big(\alpha + \sum_{i=1}^m n_i \beta_i, \zeta\big) = (\alpha , \zeta) <0,$$
	and therefore $\beta +\zeta$ is a root by \cite[p. 324 (6)]{fulton-harris:1991a}. 

\smallskip
\noindent
{\bf [iv]}
	We note that $\gamma + \beta \in \Delta_\Sigma^1$, and then use {\bf [i]} to conclude that $\gamma + \beta \in \Upsilon_\alpha^1$. 
	The $\lieg_0$-invariance of the summands $\liey_{-\alpha}$ immediately follows.
\end{proof}

\begin{proof}[Proof of Theorem \StRef{ClassificationOfMonge}]
	Suppose that  $\lieg$ is a parabolic geometry of Monge type. Then condition {\bf [i]} follows from Proposition \StRef{leader}.
	From  \EqRef{lieydecomp1} and \EqRef{lieydecomp}, we know that 
	\begin{equation}
	\label{def-liey}
	\liey = \bigoplus_{\alpha \in \Sigma\backslash\zeta}\liey_{-\alpha}.
	\end{equation}
	Since $\liey$ is abelian,  each of the summands  $\liey_{-\alpha}$  in  this decomposition must be abelian. Since $\liey_{-\alpha}$ is the $-1$-grading component for 
	 the gradation of  $\lieg(\Upsilon_\alpha)$ defined by $\alpha$, this must be a $|1|$-gradation and condition {\bf [ii]} in Theorem \StRef{ClassificationOfMonge} is established.

	Conversely,  given a  $|k|$-grading defined by $\Sigma$ such that {\bf [i]} and {\bf [ii]} hold, define  $\liey$ by   \eqref{def-liey}. By \EqRef{lieydecomp}, 
	 $\liey$ is a co-dimension $1$ subspace of $\lieg_{-1}$. 
	We now check the conditions of Definition \StRef{MongeDef} for  a parabolic Monge geometry.  
	To prove that $\liey$  is abelian, we  
	first note that  each summand $\liey_{-\alpha}$ is abelian  by hypothesis {\bf [ii]}.   Equation \EqRef{bracketliey} then proves that
	$\liey$  is abelian.  
	That  the dimension of $[\lieg_{-\zeta} , \liey]$ equals the dimension of $\liey$ follows directly from part {\bf [iii]} of Lemma \StRef{rootfacts}. 
\end{proof}

	An explicit list of parabolic geometries of Monge type can now be constructed from the classification of 	
	$|1|$-graded simple Lie algebras given in the table on page 297 of \cite{cap-slovak:2009a}.  We see that condition {\bf [ii]} of Theorem \StRef{ClassificationOfMonge} holds if and only if the graded simple algebras $\lieg(\Upsilon_{\alpha})$ are 
		$A$, $B$, $C$, $D$, $E_6$ and $E_7$ with the gradation given by a simple root at the end of its Dynkin diagram as specified in the table. 
	In the case that $\lieg(\Upsilon_\alpha)$ is of type $B_m$, $\alpha$ cannot equal $\alpha_m$ which means that
	the original $|k|$-graded algebra $\lieg$ cannot be $F_4$ with $\zeta = \alpha_4$, the short root.  Similarly, in the case that 
	$\lieg(\Upsilon_\alpha)$ is of type $C_m$, $\alpha$ cannot equal $\alpha_1$ which means that
	the original $|k|$-graded algebra $\lieg$ cannot be $C_m $ with $\zeta = \alpha_i$, for $1 \leq i \leq m-2$,  $m \geq 3$. 
	One can check that these two exceptions occur precisely when $\Sigma$ consists of just short roots. This proves the following.
\begin{Corollary} 
\StTag{LongRoot}
	Let $\lieg$ be a parabolic geometry of a simple Lie algebra as determined by the set of simple 
	roots  $\Sigma \subset \Delta^0$.  If  $\dim \lieg_{-1} > 2$, then  $\lieg$ is a parabolic geometry of Monge type if and only if
	condition {\bf [i]} of Theorem \StRef{ClassificationOfMonge} holds and  $\Sigma$ contains a long root.
\end{Corollary}
\section{Non-rigid  Parabolic Geometries of Monge type}

	Let $\lieg$ be a simple Lie algebra and let 
\begin{equation}
	\lieg = \lieg_k  \oplus \dots  \lieg_1 \oplus \lieg_0 \oplus \lieg_{-1} \oplus \dots \oplus \lieg_{-k} 
\end{equation}
	be a $|k|$-grading of $\lieg$ determined by the set of simple roots  $\Sigma \subset \Delta^0$.  We suppose that
	the parabolic geometry defined by this grading is of Monge type (see Definition  \StRef{MongeDef}) so that 
	$\Sigma$ satisifies the conditions 
	of Theorem \StRef{ClassificationOfMonge}. The purpose of this section is to determine which parabolic geometries of Monge type are non-rigid,  
	that is, we will characterize the Monge subsets of simple roots $\Sigma$ 
	with non-vanishing second  degree Lie algebra  cohomology  in positive homogeneity weight
\begin{equation}
	H^2(\lieg_{-}, \lieg)_{p}\neq 0 \quad \text{for some }p>0.
\EqTag{H2nonzero}
\end{equation}

	In a remarkable paper K. Yamaguchi  \cite{yamaguchi:1993a} gives  a complete list of all  sets of simple roots  
	for which the corresponding parabolic geometry 
	satisfies \EqRef{H2nonzero}.  Initially, we  simply  determined which of the 40 or so cases in 
	Yamaguchi's classification  were of Monge type and in this way we arrived at Theorem B.  It is a rather surprising fact that
	of all the possible sets of simple roots $\Sigma$ of Monge type,  those which are non-rigid contain either the first or the last root and 
	for the algebras  $B$, $C$,  and $D$ of rank $\geq 4$ all contain exactly  2 roots. Since these two facts alone effectively 
	reduce the proof of  Theorem B to the examination of  just a few cases and since both facts can be directly established 
	with relative ease,  we have  chosen to give the   detailed proofs here.

	We shall use Kostant's theorem \cite{kostant} to calculate the Lie algebra cohomology.\footnote{ Kostant's theorem applies more generally to
	the Lie algebra cohomology $H^q(\lieg_{-}, V)$,  where $V$ is any representation space of $\lieg$, but we limit our discussion to 
	just the case where $V = \lieg$ is the adjoint representation of $\lieg$.} To briefly describe how this calculation proceeds, we first 
	establish some standard  notation.
	Recall that we denote the set of  all roots by $\Delta$  and the  positive and negative roots by $\Delta^+$ and $\Delta^-$. 
	For a subset of simple roots $\Sigma\subset \Delta^0$, we denote by 
\begin{equation}
	\Delta^+_\Sigma=\bigcup_{k>0} \Delta_\Sigma^k
\end{equation}
	the set of roots with positive heights with respect to $\Sigma$. 

	For each simple root $\alpha_i  \in \Delta^0$ the simple Weyl reflection  $s_i$ on  the root space is defined by
	$s_i(\beta) = \beta- \langle \beta, \alpha_i\rangle\, \alpha_i$, where $\beta\in \Delta$ and $ \langle \beta, \alpha_i\rangle = \frac{2(\beta, \alpha_i)}{(\alpha_i, \alpha_i)}$.  The finite group generated by all simple Weyl reflections is the Weyl group 
	$W$ of $\lieg$.
	For any element $\sigma\in W$, we define  another set  of roots by
\begin{equation}
	\Delta_\sigma = \sigma(\Delta^-)\cap \Delta^+, 
\end{equation}
	that is, $\Delta_\sigma $ is the set of positive roots that are images of negative roots under the action of $\sigma$. 
	It is an important fact, established in many textbooks,  that if $q= \text{card} \Delta_\sigma$, 
	then $\sigma$ can be written as a product of exactly  $q$  simple Weyl reflections $s_i$, in other words,
	$\text{length}(\sigma) = \text{card} \Delta_\sigma$.
	Finally, define 
\begin{equation}
	W_\Sigma =\{\sigma\in W\,|\, \Delta_\sigma\subset \Delta^+_\Sigma\, \} \quad\text{and}\quad
	W^q_\Sigma = \{\sigma\in W_\Sigma \,|\, \text{card}\Delta_\sigma  = q \,\}.
\end{equation}
	Hasse diagrams provide an effective method for  finding the sets $W^q_\Sigma$ (see \cite[\S\S 3.2.14--16]{cap-slovak:2009a}).

	Kostant's method is based upon two key results. The first result states that the cohomology spaces $H^q(\lieg_-, \lieg)$
	are isomorphic to the the kernel of a certain (algebraic) Laplacian $\square : \Lambda^q(\lieg_-, \lieg) \to  \Lambda^q(\lieg_-, \lieg) $.
	The  forms in $\ker\square$ are said to be harmonic - they  define distinguished cohomology representatives. Since 
	$[\lieg_0, \lieg_i ]  \subset \lieg_i$,  the Lie algebra $\lieg_0$  naturally acts on  the forms  $\Lambda^q(\lieg_-, \lieg)$. 
	The second key observation is  that this action commutes with   $\square$.

	The first assertion in Kostant's theorem is that $\ker\square$ decomposes as a direct sum of irreducible representations of 
	 $\lieg_0$, each occurring with multiplicity 1,  and that there is a one-to-one correspondence between the irreducible summands
	in this decomposition and the Weyl group elements in  $W^q_\Sigma$.  For each $\sigma \in W^q_\Sigma$,  we label the corresponding
	summand by $H ^{q, \sigma}(\lieg_{-}, \lieg)$ and write	
\begin{equation}
	H^q(\lieg_-, \lieg) = \bigoplus_{\sigma \in W^q_\Sigma } H^{q, \sigma}(\lieg_{-}, \lieg).	
\end{equation} 


	
	Kostant's theorem also describes the lowest weight vector for
	$H^{q, \sigma}(\lieg_{-}, \lieg)$
	 as an irreducible $\lieg_0$-representation.\footnote{Actually Kostant \cite{kostant} studied the cohomology $H^{q, \sigma}(\lieg_{+}, \lieg)$ and gave the highest weight vector for this irreducible $\lieg_0$-representation. Through the Killing form, we have $(H^{q, \sigma}(\lieg_{+}, \lieg))^* = H^{q, \sigma}(\lieg_{-}, \lieg)$, and therefore the negative of the highest weight of the former becomes the lowest weight for the latter.} 
	 	Fix a basis $e_\alpha$ for the root space $\lieg_\alpha$, 
	and let $\omega_\alpha$ be the 1-form dual to $e_\alpha$ under the Killing form: $\omega_\alpha(x) = B(e_\alpha, x)$.  
	Let $\theta$ denote the {\deffont highest root} of $\lieg$, which is also the highest weight for the adjoint representation of $\lieg$. 
	For $\sigma \in W^q_\Sigma$, let $\Delta_\sigma = \{\beta_1, \beta_2, \ldots,  \beta_q\}$. Then 
\begin{equation}
	\omega_\sigma =  e_{-\sigma(\theta)} \otimes \omega_{-\beta_1}  \wedge  \omega_{-\beta_2}  \wedge  \dots \wedge \omega_{-\beta_q} 
\EqTag{Harmonic}
\end{equation}
is the harmonic representative for the lowest weight vector in $H^{q, \sigma}(\lieg_{-}, \lieg)$. 
	The homogeneity weight  $w_\Sigma(\omega_\sigma)$  of this form with respect to the grading
	is the homogeneity weight of all the 
	forms in $H^{q, \sigma}(\lieg_{-}, \lieg)$, since the orbit of the $\lieg_0$-action on $\omega_\sigma$ is all of $H^{q, \sigma}(\lieg_{-}, \lieg)$ and the grading element $E$ commutes with $\lieg_0$. 


\newcommand{\htSigma}{\text{\rm ht}_\Sigma}

	To calculate the  homogeneity weight  $w_\Sigma(\omega_\sigma)$ is 
			generally quite complicated but  it is possible to obtain a compact formula in the case  of immediate interest to us, 
	namely  when  $q = 2$.  Then the length of $\sigma$ is 2 and there are two simple Weyl reflections $s_i$ and $s_j$, $i \neq j$  such that
	 $\sigma=\sigma_{ij}=s_i\circ s_j$.  
\begin{Lemma}  
\StTag{Yam}
If $\sigma  = s_i \circ s_j \in W^2_\Sigma$, then $\Delta_\sigma  = \{\, \alpha_i,  s_i(\alpha_j) \,\} $, $\alpha_i\in \Sigma$, and
\begin{equation}
	w_\Sigma(\omega_\sigma)  =   -\htSigma(\theta)  + \langle \theta, \alpha_i\rangle + 1 +  
	 ( \langle \theta, \alpha_j  \rangle + 1)\,\htSigma(s_i(\alpha_j)) .
\EqTag{YamEq}
\end{equation}
	Therefore  the parabolic geometry defined by $\Sigma$ is non-rigid if  and only if 
\begin{equation}
\EqTag{YamIneq}
	\langle \theta, \alpha_i\rangle   +
	 ( \langle \theta, \alpha_j  \rangle + 1)\,\htSigma(s_i(\alpha_j))  \geq \htSigma(\theta). 
\end{equation}
\end{Lemma}
\begin{proof}   
	The  formula  \EqRef{YamEq} for the homogenity weight of  $H^{q, \sigma}(\lieg_{-}, \lieg) $
	is essentially the same as that 
given by Yamaguchi   in Section 5.3 of \cite{yamaguchi:1993a} and we follow the arguments given there.

	We first show that $	\Delta_\sigma  = \{\, \alpha_i,  s_i(\alpha_j) \}$. 
	Since  $\sigma = s_i \circ s_j$, we have
\begin{equation*}
	\sigma^{-1}(\alpha_i) = - s_j(\alpha_i) \in \Delta^{-} \quad\text{and} \quad 
	\sigma^{-1}(s_i(\alpha_j)) = s_j(\alpha_j)  = - \alpha_j \in \Delta^{-} \
\end{equation*}
	and therefore  $\alpha_i$ and $s_i(\alpha_j)$ are the two distinct elements of $\Delta_\sigma$.  
	Set $\beta_1 = \alpha_i$ and  $\beta_2  = s_i(\alpha_j)$.  The requirement $\Delta_\sigma  \in \Delta_\Sigma^+$  now implies that 
	$\alpha_i \in \Sigma$ and therefore $\htSigma(\alpha_i) = 1$.   

Since 
\begin{align*}
\sigma(\theta)& =s_i(\theta - \langle \theta, \alpha_j\rangle\alpha_j)
=\theta -    \langle  \theta, \alpha_i\rangle \alpha_i - \langle \theta, \alpha_j\rangle s_i(\alpha_j), 
\end{align*}
	we have that the weight of the harmonic representative \EqRef{Harmonic} (with $q=2$) is
\begin{align*}
		&w_\Sigma(\omega_\sigma) 
\\
 	 &= -  \htSigma(\theta) + \langle \theta, \alpha_i\rangle\, \htSigma(\alpha_i) +
	\langle \theta, \alpha_j\rangle \htSigma(s_i(\alpha_j))
	+ \htSigma(\beta_1)  + \htSigma(\beta_2),
\end{align*}
	 which reduces to   \EqRef{YamEq}.
\end{proof}

To continue, we list the expressions for $\theta$ and the nonzero $\langle \theta, \alpha_i\rangle$ for the classical Lie algebras.  
\begin{align}
A_\ell&: \theta=\alpha_1 +\alpha_2 + \cdots + \alpha_\ell, && \langle \theta, \alpha_1\rangle = \langle \theta, \alpha_\ell\rangle = 1 \EqTag{a}\\
B_\ell&: \theta=\alpha_1 +2\alpha_2 + \cdots + 2\alpha_\ell, && \langle \theta, \alpha_2\rangle =1 \EqTag{b}\\
C_\ell&: \theta=2\alpha_1 + \cdots + 2\alpha_{\ell-1} + \alpha_\ell, && \langle \theta, \alpha_1\rangle = 2 \EqTag{c}\\ 
D_\ell&: \theta=\alpha_1 +2\alpha_2 + \cdots + 2\alpha_{\ell-2} + \alpha_{\ell-1} + \alpha_\ell, && \langle \theta, \alpha_2\rangle = 1 \EqTag{d}
\end{align}
	With these formulas and  Lemma \StRef{Yam}  it is now a straightforward matter to 
	determine all the non-rigid parabolic geometries of Monge type. {\it Simply stated, the reason that there 
	are relatively  few such geometries is because  the Monge conditions in Theorem \StRef{ClassificationOfMonge} lead to a large  lower bound for
	the value of  $\htSigma(\theta)$.} 

\begin{Proposition} 
	Every  Monge parabolic geometry of type $A_\ell$ with $\ell \geq 5$ whose  simple roots $\Sigma$ are 
	interior to  the Dynkin diagram is rigid. Apart from the standard symmetry of the Dynkin diagram for $A_\ell$, the non-rigid  Monge parabolic geometries of type $A_\ell$ are  
	 $A_\ell \{\alpha_1, \alpha_2\}$ for $\ell \geq 2$ and
	$A_\ell \{\alpha_1, \alpha_2, \alpha_3\}$ for $\ell \geq 3$.
\end{Proposition}
\begin{proof} 
	If $\Sigma$ is interior to the Dynkin diagram, then by Theorem \StRef{ClassificationOfMonge} it is a connected set of 3 roots.  
	From \EqRef{a}, we have $\htSigma(\theta)=3$. Since $\alpha_1, \alpha_\ell\notin \Sigma$, 
	\EqRef{YamIneq} reduces to  
\begin{equation}
	(\langle \theta, \alpha_j \rangle  + 1)(\htSigma(\alpha_j) -  \langle \alpha_j , \alpha_i \rangle) \geq 3.
\EqTag{hope1}
\end{equation}
	But  $\langle \theta, \alpha_j \rangle \leq 1$, $\htSigma(\alpha_j) \leq 1$ and  $-\langle \alpha_j , \alpha_i \rangle \leq 1$ (from the Cartan matrix for $A_\ell$), 
	so \EqRef{hope1} is satisfied only when
\begin{equation}
	\langle \theta, \alpha_j \rangle  = 1, \quad  \htSigma(\alpha_j)  =  1, \quad  \text{and }\langle \alpha_j , \alpha_i  \rangle = -1.
\end{equation} 
	The second equation implies that $\alpha_j\in \Sigma$, which is interior to the Dynkin diagram. Then the first equation can not be satisfied by  \EqRef{a}. 
		Therefore the first statement in the proposition is established, and hence the only non-rigid cases for $A_\ell$ with $\ell\geq 5$ are $A_\ell\{\, \alpha_1, \alpha_2\, \}$ and $A_\ell\{\, \alpha_1, \alpha_2, \alpha_3\, \}$. 

	For $\ell \leq 4$ the Monge systems are $A_2\{\alpha_1, \alpha_2\}$,   $A_3\{\alpha_1, \alpha_2\}$, 
	$A_3\{\alpha_1, \alpha_2, \,  \alpha_3\}$, $A_4\{\alpha_1, \alpha_2\}$ and   $A_4\{\alpha_1, \alpha_2,   \alpha_3\}$ and hence, 
	in summary, the only possible non-rigid parabolic geometries of type $A_\ell$ are those listed in the second 
	statement of the proposition. To show that these possibilities are
	actually all non-rigid, one calculates the following table of Weyl reflections in $W_\Sigma^2$ from the Hasse diagrams and the associated weights from \EqRef{YamEq}.  
\medskip

\begin{tabular}{|l|l|l|} 
\hline
Monge Systems & $W^2_\Sigma$    & Weights of $\sigma_{ij}$
\\
\hline
$A_2\{\alpha_1, \alpha_2\}$  &$ [\sigma_{12 }, \sigma_{21}]$   & [4, 4]
\\
\hline
$A_3\{\alpha_1, \alpha_2\}$  &$ [\sigma_{12 }, \sigma_{21}, \sigma_{23}]$   & [2, 3, 1]
\\
\hline
$A_\ell\{\alpha_1, \alpha_2\}$, $\ell \geq 4$  & 	 $[\sigma_{12 },  \sigma_{21}, \sigma_{23}]$  &  [2, 3,  0]
\\
\hline
$A_3\{\alpha_1, \alpha_2, \alpha_3\}$ & $[\sigma_{12 },   \sigma_{13 }, \sigma_{21}, \sigma_{23}, \sigma_{32}]$ & [1, 1, 2, 2,  1]
\\
\hline
$A_4\{\alpha_1, \alpha_2, \alpha_3\}$ &  $[\sigma_{12 },   \sigma_{13 }, \sigma_{21}, \sigma_{23}, \sigma_{32}, \sigma_{34}]$  &  [1, 0, 2, 0, 0, 0]
\\
\hline
$A_\ell\{\alpha_1, \alpha_2, \alpha_3\}$, $\ell \geq 5$  &  $[\sigma_{12 },   \sigma_{13 }, \sigma_{21},  \sigma_{23}, \sigma_{32}, \sigma_{34}]$  &  [1, 0, 2, 0, 0,  -1]
\\
\hline
\end{tabular}

\medskip
\end{proof}

\begin{Proposition}   
	Every  Monge parabolic geometry of type $C_\ell$ with $\ell \geq 4$ for a set $\Sigma$ containing  3 simple roots is rigid.
	The non-rigid Monge parabolic geometries of type $C_\ell$ are 
	 $C_3\{\alpha_1, \alpha_2, \alpha_3\}$ and  $C_\ell\{\alpha_{\ell -1},  \alpha_\ell\}$ for $\ell \geq 3$.
\end{Proposition}
\begin{proof}  By Corollary  \StRef{LongRoot},  $\Sigma$ must contain the long simple root and therefore  
 $\Sigma=\{\alpha_{\ell-2}, \alpha_{\ell-1}, \alpha_\ell\}$ if it contains 3 simple roots.  
	Then from \EqRef{c} we have $\htSigma(\theta)=5$. Since $\ell\geq 4$, we have  $\alpha _1\notin \Sigma$. Then \EqRef{c}  shows that 
	$\langle \theta, \alpha_i\rangle =0$, and  therefore  \EqRef{YamEq} reduces to 
\begin{equation}
	(\langle \theta, \alpha_j \rangle  + 1)(\htSigma(\alpha_j) -  \langle \alpha_j , \alpha_i \rangle) \geq 5.
\EqTag{hope2}
\end{equation}
	Now we have
\begin{equation}
	\langle \theta, \alpha_j \rangle   \leq 2, \quad \htSigma(\alpha_j)  \leq 1, \quad    -\langle \alpha_j , \alpha_i \rangle \leq 2.
\end{equation}
	If  $\langle \theta, \alpha_j \rangle  =0$, then \EqRef{hope2} is not possible.  The only other possible value is $\langle \theta, \alpha_j \rangle  =2$ 
	but then   $\alpha_j = \alpha_1$ and we have  
	$\htSigma(\alpha_j)  = 0$ and $-\langle \alpha_j , \alpha_i \rangle \leq 1$ by the Dynkin diagram. 
Then \EqRef{hope2} fails again.   The first statement in the proposition is established and the   
	list of possible non-rigid parabolic geometries of type $C_\ell$ are those listed  in the second 
	statement of the proposition.  These are all non-rigid.

\medskip

\begin{tabular}{|l|l|l|} 
\hline
Monge Systems & $W^2_\Sigma $    & Weights of $\sigma_{ij} $
\\
\hline
$C_3\{\, \alpha_1,\, \alpha_2,\, \alpha_3\,\}$  &  $[\sigma_{12 },   \sigma_{13 }, \sigma_{21}, \sigma_{23}, \sigma_{32}]$ & [0, -1, 2, -1, -2]
\\
\hline
$C_3\{\, \alpha_2,\,  \alpha_3 \, \}$  &$ [ \sigma_{21}, \sigma_{23}, \sigma_{32}]$ & [1, 1, 0]
\\
\hline
$C_\ell\{\, \alpha_{\ell-1},\, \alpha_{\ell}\, \}$,  $\ell \geq 4$ & $[\sigma_{\ell-1\,\ell-2},  \sigma_{\ell-1\,\ell},  \sigma_{\ell\,\ell-1}]$ & [-1, 1, 0]
\\
\hline
\end{tabular}

\end{proof}

\begin{Proposition} Every  Monge parabolic geometry of type $B_\ell$ with $\ell \geq 4$	
	for a set $\Sigma$ containing  3 simple roots is rigid.
	The non-rigid  Monge parabolic geometries of type $B_\ell$ are
	$B_2\{\alpha_2 \}$, $B_3\{\alpha_2, \alpha_3 \}$, $B_3\{\alpha_1, \alpha_2,  \alpha_3 \}$  and
	$B_\ell\{\alpha_1, \alpha_2\}$ for $\ell \geq 2$. 
	
	Likewise, every  Monge parabolic geometry of type $D_\ell$ with $\ell \geq 4$	
	for a set $\Sigma$ containing  3 or  more simple roots is rigid. 
	The non-rigid  Monge parabolic geometries of type $D_\ell$ for $\ell \geq 4$ are $D_\ell\{\alpha_1, \alpha_2\}$.
\end{Proposition}
\begin{proof} 
	We note that for $D_\ell$, the Monge grading set $\Sigma$ can contain 4 simple roots. 
	For either $B_\ell$ or $D_\ell$ with $\ell\geq 4$, if  $\Sigma$ contains 3 or  more simple roots then, from \EqRef{b} and \EqRef{d}, we find that $\htSigma(\theta)=5$ or $6$.
	Since  
\begin{equation*}
	\langle \theta, \alpha_i \rangle \leq 1, \quad  \langle \theta, \alpha_j\rangle \leq 1, \quad \langle \theta, \alpha_i \rangle  \neq \langle \theta, \alpha_j \rangle, \quad  \htSigma(s_i(\alpha_j))   \leq 3, 
\end{equation*}	
	\EqRef{YamEq}  can only hold
	when $\langle \theta, \alpha_ i\rangle =  0$  and   $\langle \theta, \alpha_ j\rangle =  1$. In this case 
	$\alpha_j = \alpha_2$ by \EqRef{b} and \EqRef{d}, and  \EqRef{YamEq}  becomes
\begin{equation*}
	2 ( \htSigma(\alpha_2) -  \langle \alpha_2, \alpha_i\rangle) \geq 5.
\end{equation*}
	For $D_\ell$,  this is not possible because $-  \langle \alpha_2, \alpha_i\rangle\leq 1$.  For $B_\ell$, this inequality holds only if  $\alpha_2 \in \Sigma$,  $\alpha_i =\alpha_{3 }$  
	and $\ell = 3$. The first statement in the proposition for each type $B_\ell$ or $D_\ell$ is therefore established.
	
	In view of this result and Theorem  \StRef{Rank2Monge} the possible non-rigid, Monge parabolic 
	geometries of type $B_\ell$  are  $B_2\{\, \alpha_2\, \}$,   $B_2\{\, \alpha_1,\,  \alpha_2\, \}$,  $B_3\{\, \alpha_1,\,  \alpha_2\, \}$, 
	$B_3\{\, \alpha_2,\, \alpha_3\, \}$, 
	$B_3\{\,\alpha_1,\, \alpha_2, \, \alpha_3\, \}$,    $B_\ell\{\,\alpha_1,\, \alpha_2\,\}$ for $\ell \geq 4$ and   
	$B_\ell\{\, \alpha_{\ell -1},\,  \alpha_\ell\, \}$ for $\ell \geq 4$.  
	The  Monge parabolic geometries  $B_\ell\{\, \alpha_{\ell -1},\,  \alpha_\ell\, \}$ are rigid; all the others are non-rigid. 

\medskip

\begin{tabular}{|l|l|l|} 
\hline
Monge Systems & $W^2_\Sigma $    & Weights of $\sigma_{ij} $
\\
\hline
$B_2\{\alpha_1, \alpha_2\}$  &$ [\sigma_{12 }, \sigma_{21}]$   & [4, 3]
\\
\hline
$B_2\{\,\alpha_2\,\}$  &$ [\sigma_{21}]$   & [3]
\\
\hline
$B_3\{\, \alpha_1,\,\alpha_2\, \}$  & 	 $[\sigma_{12 },  \sigma_{21}, \sigma_{23}]$  &  [2, 1, 0]
\\
\hline
$B_3\{\, \alpha_2,\, \alpha_3\, \}$ & $[\sigma_{21 },   \sigma_{23 },  \sigma_{32}]$ & [-1, 0, 3]
\\
\hline
$B_3\{\, \alpha_1, \alpha_2,\, \alpha_3\, \}$ & $[\sigma_{12}, \sigma_{13}, \sigma_{21}, \sigma_{23}, \sigma_{32}]$ & [0, -3, -1, -1, 2]
\\
\hline
$B_\ell\{\, \alpha_1,\, \alpha_2\,\}$, $\ell \geq 4$  &  $[\sigma_{12 },  \sigma_{21}, \sigma_{23}]$  & [2, 1, 0]
\\
\hline
$B_4\{\, \alpha_3,\, \alpha_4\,\}$, $\ell \geq 4$  &  $[\sigma_{32 },  \sigma_{34}, \sigma_{43}]$  & [-1, -1, 0]
\\
\hline
$B_\ell\{\, \alpha_{\ell -1},\, \alpha_\ell \, \}$, $\ell \geq 5$ & $[\sigma_{\ell-1\,\ell-2},  \sigma_{\ell-1\,\ell},  \sigma_{\ell\,\ell-1}] $ & [ -2, -1, 0]
\\
\hline
\end{tabular}

\bigskip 
	
		The possible non-rigid, Monge parabolic 
	geometries of type $D_\ell$  are  $D_\ell\{\alpha_1, \alpha_2 \}$,   $D_\ell\{\alpha_{\ell-2}, \alpha_{\ell-1} \}$  and 
	 $D_\ell\{\alpha_{\ell-2}, \alpha_\ell \}$.  Note that  $D_4\{\alpha_2, \alpha_4\}$ is equivalent to  $D_4\{\alpha_1, \alpha_2\} $
	and  $D_\ell\{\alpha_{\ell-2}, \alpha_{\ell-1} \}$  and  $D_\ell\{\alpha_{\ell-2}, \alpha_\ell \}$ are equivalent for all $\ell \geq 4$. 
	For  $\ell \geq 5$ the geometries $D_\ell\{\alpha_{\ell-2}, \alpha_\ell \}$ are rigid. 

\medskip

\begin{tabular}{|l|l|l|} 
\hline
Monge Systems & $W^2_\Sigma $    & Weights of $\sigma_{ij}$
\\
\hline
$D_4\{\, \alpha_1, \,  \alpha_2\, \}$  & $[\sigma_{12 },  \sigma_{21}, \sigma_{23}, \sigma_{24} ]$  & $[ 2, 1, 0, 0]$
\\
\hline
$D_\ell\{\, \alpha_1, \,  \alpha_2\, \}$,  $\ell \geq 5$ & $[\sigma_{12 },  \sigma_{21}, \sigma_{23}]$  & $[ 2, 1, 0]$
\\
\hline
$D_5\{\, \alpha_3, \,  \alpha_{5}\, \}$  & $[\sigma_{32},\sigma_{34}, \sigma_{35},  \sigma_{53}] $ &  $[0, -1, 0, 0]$ 
\\
\hline
$D_\ell\{\, \alpha_{\ell -2}, \,  \alpha_{\ell}\, \}$,  $\ell \geq 6$  & $[\sigma_{\ell-2\,\ell-3},  \sigma_{\ell-2\,\ell - 1},  \sigma_{\ell-2\,\ell},   \sigma_{\ell\,\ell -2 }]$  & $[ -1, -1, 0, 0]$
\\
\hline
\end{tabular}

\medskip

\end{proof}	

For the exceptional Lie algebras the highest weights and  non-zero  $\langle \theta, \alpha_i \rangle$ are:
\begin{align*}
G_2 &: \theta= 3\alpha_1 + 2\alpha_2, && \langle \theta, \alpha_2\rangle =1\\
F_4 &: \theta = 2\alpha_1 + 3\alpha_2 + 4\alpha_3 + 2\alpha_4, && \langle \theta, \alpha_1\rangle =1\\
E_6 &: \theta = \alpha_1 + 2\alpha_2 + 2\alpha_3 + 3\alpha_4 + 2\alpha_5 + \alpha_6, && \langle \theta, \alpha_2\rangle =1\\
E_7 &: \theta = 2\alpha_1 + 2\alpha_2 + 3\alpha_3 + 4\alpha_4 + 3\alpha_5 + 2\alpha_6 + \alpha_7, && \langle \theta, \alpha_1\rangle =1\\
E_8 &: \theta = 2\alpha_1 + 2\alpha_2 + 4\alpha_3 + 6\alpha_4 + 5\alpha_5 + 4\alpha_6 + 3\alpha_7 + 2\alpha_8, && \langle \theta, \alpha_8\rangle =1
\end{align*}
Here the roots are labeled as in \cite[p. 454]{yamaguchi:1993a} or \cite[p. 58]{Humphreys}. 

\begin{Proposition}  The  only non-rigid  Monge parabolic geometries for the exceptional simple Lie algebras are 
$G_2 \{\alpha_1\}$ and $G_2 \{\alpha_1, \alpha_2\}$.  
\end{Proposition}
\begin{proof}  Consider first the case of $F_4$. If $\text{card}\Sigma \geq 3$, then $\htSigma(\theta) = 9$ and  with 
$\langle \theta, \alpha_i \rangle \leq 1$, $\langle \theta, \alpha_j \rangle \leq 1$, and $\htSigma(s_i(\alpha_j))  \leq 3$,
the inequality \EqRef{YamEq} cannot hold. For parabolic geometries of Monge type, $\Sigma$ must contain the long root
and this  leaves just $F_4\{\alpha_1, \alpha_2\}$ as the only possibility. But it is easy to check that this is rigid.

For $E_6$, $E_7$ and $E_8$ we have 
$\langle \theta, \alpha_i \rangle \leq 1$, $\langle \theta, \alpha_j \rangle \leq 1$,  $\htSigma(s_i(\alpha_j))  \leq 2$ and 
$\langle \theta, \alpha_i \rangle \neq \langle \theta, \alpha_j \rangle$ so that  the left-hand side of \EqRef{YamEq} does not
exceed 4. If  $\text{card}\Sigma  \geq 3$, then by the connectivity of $\Sigma$ we have $\htSigma(\theta) \geq 6, 6$ and $9$ for  $E_6$, $E_7$ and $E_8$ respectively and so 
only those geometries with  $\text{card}\Sigma  = 2$  remain as possibilities. For
$\text{card}\Sigma  = 2$ the size of $\htSigma(\theta)$ is still $\geq 5$ except for the 2 cases (apart from the symmetry of the $E_6$ Dynkin diagram) listed below, all of which are rigid by direct
calculation.

\bigskip
\begin{tabular}{|l|l|l|} 
\hline
Monge Systems & $W^2_\Sigma $    & Weights of $\sigma_{ij}$
\\
\hline
$F_4\{\, \alpha_1, \,  \alpha_2\, \}$  & $[\sigma_{12 },  \sigma_{21}, \sigma_{23}, \sigma_{24} ]$  & $[ -1, 0, -3]$
\\
\hline
$E_6\{\, \alpha_5, \,  \alpha_6\, \}$  & $[\sigma_{54 }, \sigma_{56}, \sigma_{65} ] $ &  $[-1,0,0]  $ 
\\
\hline
$E_7\{\, \alpha_6, \,  \alpha_7\, \}$  & $[\sigma_{65 }, \sigma_{67}, \sigma_{76} ] $ &  $ [-1, 0, 0] $ 
\\
\hline
\end{tabular}

\medskip
\end{proof}

We conclude this section with the description of $H^2(\lieg_-, \lieg)_p$ with positive homogeneity weights as $\lieg_0^{ss}$-representations. This  gives a characterization of the curvature
	for the normal Cartan connection which will play an important role in our subsequent study of the Cartan equivalence problem
	for non-rigid parabolic geometries  of Monge type. With this application in mind and  in view of \EqRef{IntroSym}, we will only discuss the non-rigid parabolic geometries of Monge type in Theorem B with $H^1(\lieg_-, \lieg)_p=0$ for all $p\geq 0$. Therefore we will not discuss the cases {\bf Ib}, {\bf IIIb}, and {\bf IIIa} with $\ell = 2$ in the following. 

By Kostant's theorem, the irreducible components of $H^2(\lieg_-, \lieg)_p$ are in one-to-one correspondence with $W^2_\Sigma$. The corresponding lowest weight vector is given by \EqRef{Harmonic}. We make the standard transformation from the lowest weight to the highest weight by the longest Weyl reflection.

Furthermore, if in \EqRef{Harmonic} the $e_{-\sigma(\theta)} \in \liep$, then the corresponding cohomology class is called {\deffont curvature} and otherwise it is called {\deffont torsion}. As mentioned in the introduction, all our second cohomology classes with positive homogeneities are \emph{torsion} classes, except the case for the Hilbert-Cartan equation {\bf Va}. We indicate this in our table by listing the homogeneity weight of $-\sigma(\theta)$, and it is strictly negative in all cases except one. 

In the following table, the $\omega$ are the fundamental weights of $\lieg_0^{ss}$, and the 
$V$ are the standard representations of $\lieg_0^{ss}$ corresponding to its first fundamental weight $\omega_1$. The subscript tf stands for trace free, and $\underline\otimes$ means the Cartan component of the tensor product. 

\bigskip

\hbox{
\kern - 30pt
\begin{tabular}{|l|l|l|l|l|l|l|l|}
\hline
 &Non-Rigid &  $\lieg_0^{ss}$   & $W_\Sigma^2$ & Hom.  & wts of & Highest & Rep.
\\
& Par. Monge 		&    & &   wts & $-\sigma(\theta)$ & weights    & spaces
\\
\hline
{$\bf  Ia$} & $A_3\{\alpha_1, \alpha_2, \alpha_3\}$ & $0$  & $ \sigma_{12}$, $\sigma_{32}$ & 1   & $-2$ & 0 & $\mathbb R$ 
\\
\cline{4-6}
 &  &   & $ \sigma_{13}$& 1 & $-1$ & &  
 \\
\cline{4-6}
 &  &   & $ \sigma_{21}$, $\sigma_{23}$& 2  & $-1$ & &  
\\
\cline{2-8}
 & $A_\ell\{\alpha_1, \alpha_2, \alpha_3\}$,  & $A_{\ell -3}$  & $\sigma_{12}$    & 1& $-2$ & $ \omega_1$& $V$
\\ 
\cline{4-6}
&  $\ell \geq 4$  && $\sigma_{21}$  & 2 & $-1$ &$ $&
\\
\hline
{$\bf IIa$}& $C_3\{\alpha_2, \alpha_3\}$ & $ A_1$  & $ \sigma_{21}$ & 1 &$-1$ &  0 &  $\mathbb R$ 
\\
\cline{4-8}
&&& $\sigma_{23}$  &1  & $-3$ & $5\omega_1$ & $S^5(V)$
\\
\cline{2-8}
& 
\parbox[t][33pt][b]{10pt}{\begin{align*} &C_\ell\{\alpha_{\ell-1}, \alpha_\ell\}\\ &\ell \geq 4
 \end{align*} }
 & $ A_{\ell -2}$  & $ \sigma_{\ell-1, \ell}$ & 1 & $-3$ & $3\omega_1 +  2\omega_{\ell -2}$ 
& $S^3(V) \underline{\bigotimes} S^2(V^*)$
\\[-3\jot]
\hline
{$\bf IIb$} & $C_3\{\alpha_1, \alpha_2, \alpha_3\}$ & $ 0$  & $ \sigma_{21}$ & 2 & $-1$ &  0  & $\mathbb R$ 
\\
\hline
{$\bf IIIa$} & $B_3\{\alpha_1, \alpha_2\}$ & $ A_1$  & $ \sigma_{12}$ & 2 & $-1$ & $4\omega_1$ & $S^4(V)$
\\
\cline{4-8}
&&& $\sigma_{21}$  &1 & $-2$ &  $6\omega_1$ & $S^6(V)$
\\
\cline{2-8}
& $B_\ell\{\alpha_1, \alpha_2\}$ & $ B_{\ell -2}$  & $ \sigma_{12}$ & 2 & $-1$ &  $2\omega_1$ & $S^2(V)_{\text{tf}}$
\\
\cline{4-8}
& $\ell \geq 4$ && $\sigma_{21}$  &1 & $-2$ &  $3\omega_1$ &   $S^3(V)_{\text{tf}}$
\\
\hline
$\bf IIIc$ & $B_3\{\alpha_2, \alpha_3\}$ & $ A_1$  & $ \sigma_{32}$ & 3 & $-1$ &  $2\omega_1$ 
& $S^2(V)$
\\
\hline
{$\bf IIId$} & $B_3\{\alpha_1, \alpha_2, \alpha_3\}$ & $ 0$  & $ \sigma_{32}$ & 2 & $-2$ & 0 & ${\mathbb R}$
\\
\hline
 {$\bf IVa$}& $D_4\{\alpha_1, \alpha_2\}$ & $ A_1 \oplus A_1$  & $ \sigma_{12}$ & 2 & $-1$ & $ [2\omega_1, 2\omega_1]$ 
&  $S^2(V_1) \otimes S^2(V_2) $
\\
\cline{4-8}
&&& $\sigma_{21}$  &1 &$-2$ & $ [3\omega_1, 3\omega_1]$  &  $S^3(V_1) \otimes S^3(V_2) $
\\
\cline{2-8}
& $D_\ell\{\alpha_1, \alpha_2\}$ & $ D_{\ell -2}$  & $ \sigma_{12}$ & 2 &$-1$ & $2\omega_1$ & $S^2(V)_{\text{tf}}$
\\
\cline{4-8}
&$\ell \geq 5$&& $\sigma_{21}$  & 1 &$-2$ &  $3\omega_1$  & $S^3(V)_{\text{tf}}$
\\
\hline
 {$\bf V{a}$} & $G_2\{\alpha_1\}$ & $A_1$ & $\sigma_{12}$  & 4 &{\bf 0} & $4\omega_1$ & $S^4(V)$
\\
\hline
{$\bf V{b}$} & $G_2\{\alpha_1, \alpha_2\}$ & 0 & $\sigma_{12}$  & 4 & $-1$ & 0 & ${\mathbb R}$
\\
\hline
\end{tabular}
}

\section{Standard Differential Systems for the Non-rigid  Parabolic Geometries of Monge type}
	In this section we use the standard matrix representations of the classical simple Lie algebras to explicitly 
	calculate the structure equations for each negatively  graded component  $\lieg_{-}$ of the non-rigid parabolic geometries 
	of Monge type enumerated in Theorem B.
	We give the structure equations in terms of the  dual 1-forms. 
	In each case these structure equations are easily integrated to give the
	Maurer-Cartan forms on the nilpotent Lie group $N$ for the Lie algebra $\lieg_{-}$ and the  associated 
	standard differential system is found.

\medskip 
\noindent
{\bf Ia. $\boldsymbol{A_\ell\{\, \alpha_1,\,  \alpha_2,\,   \alpha_3\, \},\ \ell \geq 3.$ }}
	We use the standard matrix representation for the Lie algebra $A_\ell=\lies\liel(\ell+1)$. 
	Then the Cartan subalgebra is defined by the  trace-free diagonal matrices $H_i = E_{i,i} - E_{i+1, i+1}$, $1 \leq i \leq \ell$.
	Let $L_i$ be the linear function on the Cartan subalgebra taking  the value of the $i$th entry.
	The simple roots are  $\alpha_i=L_i-L_{i+1}$ for $1\leq i\leq \ell$  and the positive roots are 
	$\alpha_i +\dots +\alpha_j$ for $1 \leq i \leq j \leq \ell$. Thus the positive roots of height 1 with respect 
	to  $\Sigma = \{\, \alpha_1,\,  \alpha_2,\,   \alpha_3\, \}$ are  $\alpha_1$, $\alpha_2$  and $\alpha_3 + \cdots +\alpha_i$ for 
	$ 3\leq i \leq \ell$.
	The  leader is  $X = e_{-\alpha_2} = E_{3,2}$ and the remaining root vectors of height $-1$, which define a basis for the abelian 
	subalgebra  $\liey$ 
	are $P_0 =  e_{-\alpha_1} = E_{2,1}$ and  $P_i =  e_{-\alpha_3 - \cdots- \alpha_{i+2}}  = E_{i+3, 3}$ for $ 1 \leq i \leq \ell- 2$. 
	This somewhat obscure labeling of the basis vectors will be justified momentarily.
	It is easy to verify that the given matrices are indeed the required root vectors with 
	respect to the above choice of  Cartan subalgebra.
	These vectors define the weight $-1$ component $ \lieg_{-1}$ of the grading for $\lies\liel(\ell+1)$ defined by $\Sigma$. 
	Since $[\lieg_{-1}, \lieg_{-i}] = \lieg_{-i -1}$, we calculate the remaining vectors in $\lieg_{-}$ to be
\begin{gather*}
	[ \,P_0,\, X\,] = Y_0 = -E_{3,1}, \quad   [\,P_i,\,  X\,] = Y_i =  E_{i+3,2}, 
\\
	[\, P_0, \, Y_i\,] = [ \,P_i, \, Y_0\,]  = Z_i = -E_{i+3, 1} .
\end{gather*}
The grading  of $\lieg_{-}$ and full structure equations are therefore
\newcolumntype{C}{>$ c < $} 
\begin{gather*}
\begin{aligned}
\lieg_{-1} &= \langle\,  P_0, \,  P_1,\, P_2, \,\ldots \, ,  P_{\ell-2},\, X\, \rangle, \\
\lieg_{-2} &= \langle\,  Y_0, \,  Y_1, \ldots, Y_{\ell-2}\,\rangle,\\
\lieg_{-3} &= \langle\,  Z_1, \ldots,  Z_{\ell -2} \,\rangle,
\end{aligned}
\quad\text{and} 
\quad
\begin{tabular}{C|CCCCCC}
	& P_0 & P_i &  X  & Y_0& Y_i & Z_i \\
\hline 
P_0 &0&  0  &   Y_0  & 0  & Z_i &  0 \\
P_i  &  & 0& Y_i & Z_i&0 &0 \\
X  & &  & 0 & 0  & 0&  0\\
Y_0 &&&& 0 & 0& 0 \\
Y_i  &&&&& 0 & 0\\
Z_i &&&&&& 0\\
\end{tabular} .
\end{gather*}

	In terms of the dual basis  $\{\, \theta_{p}^0,\, \theta_{p}^i,\,  \theta_x,\, \theta_{y}^0,\,  \theta_{y}^i,\, \theta_{z}^i\, \}$ 
	for the Lie algebra these structure equations  are
\begin{gather*}
d\theta_{p}^0 = 0, \quad  d\theta_p^i = 0 \quad  d\theta_{x} = 0, \\
d\theta_y^0 = \theta^{x} \wedge \theta_{p}^0, \quad  d\theta_y^i = \theta_{x} \wedge \theta_{p}^i, \quad  
d\theta_z^i  = \theta_y^i  \wedge \theta_{p}^0 +  \theta_y^0\wedge \theta_{p}^i.
\end{gather*}
	These structure equations are easily  integrated to give the following Maurer-Cartan forms on the nilpotent Lie group
\begin{gather*}
\theta_{p}^0  = dp^0, \quad  \theta_{p}^i = d p^i,  \quad\theta_x = dx,  \quad
\theta_y^0 = dy^0  - p^0 dx, \quad	 \theta_y^i = dy^i  - p^i dx,  \\
\quad  
\theta_{z}^i  = dz^i - p^0 dy^i  -  p^i dy^0  +   p^0 p^i   dx.
\end{gather*}
	The standard Pfaffian system defined by the parabolic geometry  $A_\ell\{\,  \alpha_1, \, \alpha_{2} ,\,  \alpha_3\,\}$ is  therefore 
\begin{align*}
	I_{A_\ell\{1,2,3\}} 
& 	=   \text{span\,} \{\,\theta_y^0, \,  \theta_y^i, \, \theta_{z}^i \, \}   
\\
&	= \text{span\,} \{\, dy^0  - p^0 dx, \, dy^i  - p^i dx,\,  dz^i   -   p^0 p^i   dx\,  \} .
\end{align*}
	This is the canonical Pfaffian  system for the Monge equations \EqRef{StandardA}.  
	By  Tanaka's theorem we are guaranteed that the symmetry algebra of the system is $\lies\liel(\ell + 1)$.
\par
\medskip 
\noindent
{\bf IIa. $\boldsymbol{C_\ell \{ \alpha_{\ell -1},\, \alpha_\ell \}, \  \ell \geq 3.}$ }
	The split real form for $C_\ell$  which we shall  use is  
	$\lies\liep(\ell, {\mathbb R}) = \{\, X \in \lieg\liel(2\ell, {\mathbb R}) \, |\,  X^tJ + J X = 0\, \}$, where 
\begin{equation*}
	J = \bmatrix 0 & K_\ell \\ - K_\ell & 0 \endbmatrix  \quad  \text{and}  \quad K_{\ell} =\begin{bmatrix} & & 1\\ & \iddots & \\1 & & \end{bmatrix}.  
\end{equation*}
	Each $X \in \lies\liep(\ell, {\mathbb R})$ may be written as $X = \displaystyle \bmatrix A & B \\ C &- A'\endbmatrix$ where 
	$A$,  $B$, $C$ are $\ell \times \ell$ matrices, $A' = KA^tK$   and $B =B'$ and $C= C'$. 
	The diagonal matrices $H_i =  E_{i,i} - E_{2\ell +1-i, 2\ell +1-i} $ define a Cartan subalgebra. The
	simple roots $\alpha_i = L_i - L_{i+1}$,  $ 1 \leq i \leq \ell-1$ and $\alpha_\ell = 2L_\ell$  
	and the positive roots  are
\begin{equation}
\begin{cases}
	&\alpha_i +\dots + \alpha_{j-1} \quad \text{ for }  1\leq i < j  \leq \ell , \quad\text{and}  \\ 
	&( \alpha_i +\dots +\alpha_{\ell-1}) + ( \alpha_j +\dots +\alpha_{\ell})\quad  \text{for } 1\leq i \leq j  \leq \ell .
\end{cases}
\end{equation}
(For the Lie algebras of type $B$, $C$ and $D$, we use the lists of positive roots from \cite{varadarajan:1984}.) 
	Therefore, for the choice of simple roots $\Sigma =  \{ \alpha_{\ell -1},\, \alpha_\ell \}$, the roots of height 1 are
	$\alpha_\ell$ and $\alpha_i + \dots +\alpha_{\ell- 1}$, for  $1\leq i \leq  {\ell -1} $. The root $-\alpha_\ell$ is our leader
	with root vector  $X =  E_{\ell +1, \ell}$.  A basis for the  abelian subalgebra $\liey$, corresponding to the remaining roots of height
	$-1$ is given by $P_i = E_{\ell, i} - E_{2\ell +1 - i , \ell+1}$.  One easily checks that these matrices belong to $\lies\liep(\ell, {\mathbb R})$
	and that they are indeed  root vectors for  the above choice of Cartan subalgebra. By direct calculation we then  find that 
\begin{gather*}
	[ \,P_i, \, X\,] = Y_i = -E_{\ell+1,i} - E_{2\ell +1 -i,\ell},\quad\text{and}   
\\
	[\, P_i,\,  Y_i\,] =  2Z_{ii} = 2E_{2\ell +1 -i,i}   \quad \text{and } [\, P_i,\,  Y_j\,] = Z_{ij}= E_{2\ell +1 -i,j} + E_{2\ell +1 -j,i} 
\end{gather*}
	Note that $Z_{ij} = Z_{ji}$. The grading and full structure equations for $\lieg_{-}$  are therefore
\begin{gather*}
\begin{aligned}
\lieg_{-1} &= \langle\,  P_1,\, P_2, \,\ldots \, ,  P_{\ell-1}, \, Z\, \rangle, \\
\lieg_{-2} &= \langle\, \,  Y_1, \ldots, Y_{\ell-1}\, \rangle,  \\
\lieg_{-3} &= \langle\,  Z_{11},\, Z_{12} \,\,\ldots\, , Z_{\ell\ell} \, \rangle,
\end{aligned}
\quad\text{and} 
\quad
\begin{tabular}{C|CCCCCC}
	& P _i &  X  & Y_i & Z_{ij} \\
\hline 
P_h&  0  & Y_h  & \epsilon Z_{hi}  & 0 \\
X  &     & 0 & 0&0  \\
Y_k  & &   & 0&  0\\
Z_{hk} && & & 0 
\end{tabular}
\end{gather*}
	where $\epsilon = 2$ if $i =j$ and $\epsilon = 1$ otherwise.	In terms of the dual basis  
	$\{\, \theta_{p}^i,\,  \theta_x,\,  \theta_{y}^i,\,  \theta_{z}^{ij}\}$ for $\lieg_{-}$ 
	these structure equations  are
\begin{gather*}
	d\theta^i_p = 0 \quad  d\theta_{x} = 0,\quad  d\theta_y^i = \theta_{x} \wedge \theta_{p}^i, \quad  
\\
	d\theta_z^{\,ij}  = \theta_y^i  \wedge \theta_p^j +  \theta_y^j \wedge \theta_p^i
\end{gather*}
	These structure equations are easily  integrated to give the following Maurer-Cartan forms
\begin{gather*}
	\theta_{p}^i = d p^ i,  \quad \theta_x = dx,  \quad 	 \theta_y^i = dy^i  - p^i dx,
\\
\theta^{ij}_z  = dz^{ij}- p^i dy^j  -  p^j dy^i   + p^ip^j\,dx.
\end{gather*}
	The standard Pfaffian system defined by the parabolic geometry  $C_\ell\{\, \alpha_{\ell-1},\,  \alpha_\ell \,\}$
	is  therefore
\begin{align*}
	I_{C_\ell\{\ell -1, \ell\}}
&	= \text{span\,} \{\,  \theta^i_y, \, \theta^{ij}_z  \,\}
  	=   \text{span} \{\, dy^i  - p^i dx,\,  dz^{ij}   -    p^ip^jdx\,  \} .
\end{align*}
	This is the canonical Pfaffian  system for the Monge equations \EqRef{StandardC}.  
	By  Tanaka's theorem we are guaranteed that the symmetry algebra of the system is $\lies\liep(\ell, {\mathbb R})$.
\par
\medskip 
\noindent
{\bf IIIa. $\boldsymbol{B_\ell \{ \alpha_1,\, \alpha_2\},\ \ell \geq 3.} $ } 
	The split real form for $B_\ell$ is  $\lies\lieo(\ell +1,  \ell)$  which we take to be 
	the Lie algebra of $n \times n$ matrices,  $n = 2\ell +1$, which are skew-symmetric with respect to the anti-diagonal matrix  
	$K_{n} = [k_{ij}]$.
	 The diagonal matrices $H_i = E_{i,i} - E_{n +1 -i, n +1  -i}$ define a Cartan subalgebra. The
	 simple roots are $\alpha_i = L_i - L_{i+1}$,  $ 1 \leq i \leq \ell-1$ and $\alpha_\ell = L_\ell$
	and the positive roots  are
\begin{equation*}
\begin{cases}
	&\alpha_i +\dots + \alpha_{j} \quad \text{ for }  1\leq i \leq j  \leq \ell , \quad\text{and}  \\ 
	&( \alpha_i +\dots +\alpha_{\ell}) + ( \alpha_j +\dots +\alpha_{\ell})\quad  \text{for } 1\leq i < j  \leq \ell.
\end{cases}
\end{equation*}
	Therefore, for the choice of simple roots $\Sigma =  \{ \alpha_{1},\, \alpha_2\}$, the roots of height $1$ are
\begin{equation*}
\begin{cases}
	&\alpha_1 \\
	& \alpha_2+ \dots +\alpha_{j} \quad \text{for } 2\leq j \leq {\ell},  \quad \text{and} \\
	& \alpha_2+\dots + \alpha_{i - 1} + 2\alpha_i +\dots + 2\alpha_{\ell} \quad \text{for } 3\leq i\leq \ell.  
\end{cases}
\end{equation*}

	The root $-\alpha_1$ is our leader
	with root vector  $X =  E_{2, 1} - E_{n, n -1}$.  
	A basis for the  abelian subalgebra $\liey$, corresponding to the remaining roots of height
	$-1$ is given by $P_i = E_{i+2, 2} - E_{n - 1, n - i -1}$ for $1 \leq i \leq n-4$.  
	One easily checks that these matrices belong to $\lies\lieo(\ell +1, \ell)$
	and that they are indeed  root vectors for the above choice of Cartan subalgebra.
	By direct calculation we find that
	for  $1 \leq i \leq n-4$ and $1 \leq j \leq n-4$
\begin{gather*}
	[ \,P_i,\, X\,] = Y_i = E_{i+2, 1} - E_{n, n- i-1},\quad\text{and}   
\\
[\,P_i,\,  Y_j\,] =  \kappa_{ij}Z,  \quad \text{where} \quad Z = E_{n, 2} -E_{n-1, 1}  \text{ and } [\kappa_{ij}] = K_{n-4}.
\end{gather*}
	The grading and full structure equations for $\lieg_{-}$ are therefore
\begin{gather*}
\begin{aligned}
\lieg_{-1} &= \langle\,  P_1,\, P_2, \,\ldots \, ,  P_{n-4}, \, X\, \rangle, \\
\lieg_{-2} &= \langle\, \,  Y_1, \ldots, Y_{n-4}\rangle, \\
\lieg_{-3} &= \langle \,Z \, \rangle,
\end{aligned}
\quad\text{and} 
\quad
\begin{tabular}{C|CCCCCC}
	& P_i &  X  & Y_i & Z \\
\hline 
P_h&  0  & Y_h  & \kappa_{hi}Z  & 0 \\
X   &     & 0 & 0&0  \\
Y_k  & &   & 0&  0\\
Z && & & 0 
\end{tabular} .
\end{gather*}
	
	In terms of the dual basis  $\{\, \theta_{p}^i,\,  \theta_x,\,  \theta_{y}^i,\,  \theta_{z}\,\}$ for  $\lieg_{-}$
	the structure equations  are
\begin{gather*}
	d\theta^i_p = 0, \quad  d\theta_{x} = 0,\quad  d\theta_y^i = \theta_{x} \wedge \theta_{p}^i, \quad  
	d\theta_z  =  \kappa_{ij}\theta_y^j  \wedge \theta_p^i.
\end{gather*}
	which are integrated to give the following  Maurer-Cartan forms
\begin{gather*}
	\theta_{p}^i = d p^i,  \quad \theta_x = dx,  \quad 	 \theta_y^i = dy^i  - p^i dx,
\\
\theta_z  = dz - \kappa_{ij}p^i dy^j   + \frac12 \kappa_{ij}{p^ip^j}\,dx.
\end{gather*}
	The standard Pfaffian system defined by the parabolic geometry $B_\ell\{\, \alpha_1, \, \alpha_2 \,\}$  is  therefore 
\begin{equation}
	I_{B_\ell\{1, 2\}} =   \text{span\,} \{\,  \theta^i_y, \, \theta_z  \,\} = \text{span\,} \{\, dy^i  - p^i dx,\,  dz   -    \frac12 \kappa_{ij}p^ip^jdx\,  \} .
\end{equation}
	This is the canonical Pfaffian  system for the Monge equations \EqRef{StandardB}.  
	By  Tanaka's theorem we are guaranteed that the symmetry algebra of the system is $\lies\lieo(\ell +1, \ell)$.
	
\par
\medskip
\noindent
{\bf IVa. $\boldsymbol{D_\ell\{ \alpha_1,\, \alpha_2\},\ \ell \geq 4.} $ } 
In this case $ n = 2\ell$  and the positive roots are 
\begin{equation}
\begin{cases}
	&\alpha_i +\dots + \alpha_{j-1} \quad \text{ for }  1\leq i < j  \leq \ell,  \quad\text{and}  \\ 
	&( \alpha_i +\dots +\alpha_{\ell-2}) + ( \alpha_j +\dots +\alpha_{\ell})\quad  \text{for } 1\leq i < j  \leq \ell 
\end{cases}
\end{equation}
	but otherwise the formulas from {\bf III}	 remain unchanged. 

\medskip

We now turn to the exceptional cases.

\medskip
\noindent
{\bf Ib. $\boldsymbol{A_\ell \{ \alpha_1,\, \alpha_2\, \}, \ \ell \geq 2 }.$}
	We  retain the notation used in {\bf Ia}. In the present case the leader is  $X = e_{-\alpha_1}  = E_{2, 1}$ and the 
	matrices  $P_i =  e_{-\alpha_2 - \cdots - \alpha_{i +2}} = E_{i +2, 2}$,  $1 \leq i  \leq \ell-1$  define a basis for $\liey$. The structure
	equations are  $[\, P_i,  X\, ] = Y_i = E_{i +2, 1}$  and the standard differential system is the contact system
\begin{equation}
	I_{A_\ell \{\alpha_a, \alpha_2\}}  = \{\, dy^1 - p^1 dx, dy^2 - p^2 dx,  \ldots, dy^{\ell -1} - p^{\ell -1} dx\, \}
\end{equation}
	on the jet space $J^1({\mathbb R}, \mathbb R^{\ell -1})$.

\medskip
\noindent
{\bf IIb. $\boldsymbol{C_3\{ \alpha_1,\, \alpha_2, \alpha_3\, \}}$.} 
	The roots of height 1 are  $\Sigma = \{ \alpha_1, \alpha_2, \alpha_3 \}$ .
	The root $-\alpha_2$ is our leader
	with root vector  $X =  E_{3, 2} - E_{5, 4}$.  
	A basis for the  abelian subalgebra $\liey$, corresponding to  the roots $-\alpha_1$ and $-\alpha_3$, is
	$P_1 = E_{2, 1} - E_{6, 5}$  and $P_2 = E_{4, 3}$ and we calculate 
\begin{alignat*}{3}	
	Y_1 &= [P_1, X] =  E_{6,4} -E_{3,1}, &\quad  Y_2  &= [P_2, X] = E_{5,3} + E_{4,2},
\\
        Z_1 &= [P_1, Y_2] = -E_{4,1} -E_{6,3}, &\quad Z_2 &=[X, Y_2] = -2E_{5,2}, 
\\
         Z_3 &=[X, Z_1]  =  E_{6,2} + E_{5,1}, &\quad   Z_4 & = [P_1, Z_3 ]  = -2E_{6,1}.
\end{alignat*}
The grading and full structure equations for $\lieg_{-}$ are therefore
\begin{gather*}
\begin{aligned}
\lieg_{-1} &= \langle\,  P_1,\, P_2, \, X\, \rangle, \\
\lieg_{-2} &= \langle\, \,  Y_1,\, Y_2\, \rangle, \\
\lieg_{-3} &= \langle \, Z_1, \, Z_2 \, \rangle, \\
\lieg_{-4} &= \langle \,Z_3 \, \rangle, \\
\lieg_{-5} &= \langle \,Z_4 \, \rangle \\
\end{aligned}
\quad\text{and} 
\quad
\begin{tabular}{C|CCCCCCCCCC}
	& P_1 &P_2 &  X  & Y_1 & Y_2 & Z_1 &Z_2 & Z_3 & Z_4 \\
\hline 
P_1 &  0 & 0  & Y_1  & 0      & Z_1    & 0     &2 Z_3 & Z_4 & 0  \\
P_2 &     & 0  & Y_2  & Z_1  & 0    & 0 &    0  & 0 & 0  \\
X     &     & &0 & 0& Z_2 &Z_3 &0&0 &0 \\
Y_1 &     & & &0 & Z_3&Z_4 &0  &0 &0 \\
Y_2 &     & & & & 0&0&0  &0 &0 \\
Z_1 &&&&&& 0 &0 &0&0 \\
Z_2 &&&&&&&0 &0 &0\\
Z_3 &&&&&&& &0 &0 \\
Z_4 &&&&&&&&& 0 \\
\end{tabular} .
\end{gather*}
	In terms of the dual basis  $\{\theta_{p}^1,\,  \theta_{p}^2, \, \theta_x,\,  \theta_{y}^1,\,\theta_{y}^2,\, 
	\theta_{z}^1\,\,   \theta_{z}^2,\,   \theta_{z}^3,\,  \theta_{z}^4\}$ for  $\lieg_{-}$ the structure equations  are
\begin{gather*}
	d\theta^1_p = 0, \quad   d\theta^2_p = 0, \quad d\theta_{x} = 0,\quad  d\theta_y^1 = \theta_{x} \wedge \theta_{p}^1, 
	\quad   d\theta_y^2 = \theta_{x} \wedge \theta_{p}^2,
\\
	d\theta^1_z  = \theta_y^1 \wedge \theta_p^2 + \theta_y^2  \wedge \theta_p^1, \quad d\theta^2_z = \theta_y^2  \wedge \theta_x, 
\\ 
	d\theta^3_z  = -\theta_y^1 \wedge \theta_y^2 + \theta^1_z\wedge \theta_x+ 2\theta^2_z\wedge \theta^1_p, \quad
	d\theta^4_z  = \theta_z^1 \wedge \theta_y^1 + \theta^3_z \wedge \theta_p^1. 
\end{gather*}
	which integrate to give
\begin{gather*}
	\theta^1_p = dp^1, \quad \theta^2_p= dp^2,\quad \theta_{x}  = dx, \quad  \theta_y^1 =dy^1 - p^1 dx, \quad \theta_y^2= dy^2 - p^2 dx, 
\\
	\theta^1_z  =  dz^1 - p^2dy^1 - p^1dy^2 + p^1p^2dx, \quad	 \theta^2_z  =  dz^2 -x dy^2,\quad 
\\
	\theta^3_z  = dz^3 -xdz^1 -2p^1dz^2 + (2xp^1 - y^1)dy^2 \quad 
\\
\theta^4_z  =   dz^4  + (xp^1 -y^1)dz^1 + (p^1)^2dz^2 - p^1dz^3  -p^1(xp^1 - y^1)dy^2. \quad	
\end{gather*}
	The standard  differential system  for $C_3\{\alpha_1, \alpha_2, \alpha_3\}$ is therefore
\begin{align*}
	I_{C_3\{1,2, 3\}} 
&	= \{\theta^1_p,\,  \theta^2_p, \, \theta^1_z,\,  \theta^2_z, \,  \theta^3_z, \, \theta^4_z \}
\\
&	= \{ dy^1 - p^1dx, dy^2 - p^2dx,\, dz^1 - p^1 p^2 dx,\, dz^2 - xp^2 dx, \, 
\\
&	\quad \quad dz^3 - (y^1p^2 + xp^1p^2)\, dx, \,dz^4 - y^1p^1p^2 dx \}	
\end{align*}
	is the canonical differential system for the first order Monge system \eqref{exceptionalC}.

\medskip
\noindent
{\bf IIIb. $\boldsymbol{ B_2 \{\, \alpha_2\, \}}$.}  The roots of height 1 are $\alpha_2$ and $\alpha_1 + \alpha_2$, and
the  standard differential system is just the canonical differential system 
\begin{equation}
		I_{ B_2 \{\, \alpha_2\, \}}  = \{dy  - p\,dx\}
\end{equation}	
\par
\medskip
\noindent
{\bf IIIc. $\boldsymbol{ B_3\{\, \alpha_2,\, \alpha_3\, \}}$.} 
	The roots of height 1 are  $\Sigma = \{\, \alpha_1 + \alpha_2,\,  \alpha_2,\, \alpha_3 \,\}$. 
	The root $-\alpha_3$ is our leader
	with root vector  $X =  E_{4, 3} - E_{5, 4}$.  
	A basis for the  abelian subalgebra $\liey$, corresponding to  the roots $-\alpha_2$ and $- \alpha_1 -\alpha_2$, is
	$Q_1 = E_{3, 2} - E_{6, 5}$  and $Q_2 = E_{3, 1} - E_{7, 5}$ and we calculate 
\begin{alignat*}{2}
	P_1 &= [Q_1, X] \ = E_{6,4} - E_{4, 2}, &\quad  P_2 &= [Q_2, X] =  E_{7,4} - E_{4,1} ,\\ 
	Y_1 &= [P_1, X]\, \ =  E_{6,3} - E_{5, 2}, & Y_2 &=[P_2, X]  = E_{7,3} - E_{5,1}, \\
	   Z &= [Q_1, Y_2]  =  E_{6, 1} - E_{7,2}.
\end{alignat*}
The grading and full structure equations for $\lieg_{-}$ are therefore
\begin{gather*}
\begin{aligned}
\lieg_{-1} &= \langle\,  Q_1,\, Q_2, \, X\, \rangle, \\
\lieg_{-2} &= \langle\,  P_1, \,P_2\, \rangle, \\
\lieg_{-3} &= \langle \,Y_1, \, Y_2 \, \rangle, \\
\lieg_{-4} &= \langle \,Z \, \rangle, \\
\end{aligned}
\quad\text{and} 
\quad
\begin{tabular}{C|CCCCCCCCC}
	& Q _1 &Q_2 &  X  & P_1 & P_2 & Y_1 &Y_2 & Z  \\
\hline 
Q_1 &  0 & 0  & P_1  & 0      & 0  & 0     & Z& 0   \\
Q_2 &     & 0  & P_2  & 0  & 0    & -Z&    0  & 0  \\
X     & &   0 & 0 &-Y_1 & -Y_2&0 &0 &0\\
P_1 &     & & &0 & -Z&0 &0&0  \\
P_2 &     & & & & 0&0 &0  &0 \\
Y_1 &&&&&& 0 &0 &0\\
Y_2 &&&&&&&0 &0 \\
Z &&&&&&&& 0  \\
\end{tabular} .
\end{gather*}
	In terms of the dual basis  
	$\{\,\theta_{q}^1,\,  \theta_{q}^2, \, \theta_x,\,  \theta_{p}^1,\,\theta_{p}^2,\, \theta_y^1\,\,   \theta_y^2,\,   \theta_z\}$ 
	for  $\lieg_{-}$  the structure equations are
\begin{gather*}
	d\theta^1_q = 0, \quad   d\theta^2_q = 0, \quad d\theta_{x} = 0,\quad  d\theta_p^1 = \theta_{x} \wedge \theta_{q}^1,
	\quad   d\theta_p^2 = \theta_{x} \wedge \theta_{q}^2,
\\
	d\theta^1_y  =  \theta_x \wedge  \theta_p^1 , \quad d\theta^2_y = \theta_x  \wedge  \theta_p^2  , \quad
	d\theta_z =   - \theta^1_y\wedge \theta^2_q  + \theta^2_y\wedge \theta^1_q + \theta^1_p \wedge \theta^2_p,
\end{gather*}
	and one finds that
\begin{gather*}
	\theta^1_q= dq^1, \quad \theta^2_q= dq^2\quad \theta_x = dx, \quad  \theta_p^1 =dp^1 - q^1 dx, \quad \theta_p^2= dy^2 - q^2 dx, 
\\
	\theta^1_y  =  dy^1 - p^1dx, \quad	 \theta^2_y  = dy^2 -p^2dx,  \quad 
\\
	\theta_z  = dz   - p^2dp^1 +q^2dy^1 - q^1 dy^2  +(p^2q^1 - p^1q^2)dx.
\end{gather*}
	The stantard Pfaffian differential system  for $B_2\{\alpha_2, \alpha_3\}$ is  therefore
\begin{align*}
	I_{B_2\{2, 3\} }
	& = \{\, \theta_y^1,\, \theta_y^2,\, \theta^1_p, \, \theta^2_p,\,  \theta^1_q, \, \theta^2_q, \theta_z\,\}
\\
	& =  \{\,dy^1 - p^1dx,\, dy^2 - p^2dx,\, dp^1 - q^1dx,\, dp^2 - q^2dx,\, dz - p^2q^1 dx\, \}
\end{align*}
	which coincides with the differential system for  the Monge equations \eqref{excpB1}.  We remark that 
	this Monge system  may also be encoded on an 7-dimensional manifold  by the Pfaffian system
	$\{\, \theta_y^1,\, \theta_y^2,\, \theta^1_p, \, \theta^2_p, \theta^2_q, \theta_z\,\}$ -- however,
	the symmetry algebra of this latter Pfaffian  system is only 16-dimensional.

\medskip
\noindent
{\bf IIId. $\boldsymbol{ B_3, \Sigma =  \{ \alpha_1,\, \alpha_2, \alpha_3\, \}}$. } 
	The roots of height 1 are  $\Sigma = \{ \alpha_1, \alpha_2, \alpha_3 \}$. 
	The root $-\alpha_2$ is our leader
	with root vector  $X =  E_{3, 2} - E_{6, 5}$,
	a basis for the  abelian subalgebra $\liey$, corresponding to  the roots $-\alpha_1$ and $-\alpha_3$,  is 
	$P_1 = E_{2, 1} - E_{7, 6}$  and $P_2 = E_{4, 3} - E_{5, 4}$ and we calculate 
\begin{alignat*}{3}
	Y_1 &= [P_1, X] = E_{7, 5} - E_{3, 1}, 	&\quad  Y_2 &=  [P_2, X] = E_{4,2} - E_{6,4}, 
\\ \ Z_1 &= [P_1, Y_2] = E_{7,4} - E_{4,1} ,  	&\quad  Z_2 &=  [P_2, Y_2]  = E_{6,3} - E_{5,2},
\\\   Z_3 &= [P_1, Z_2]  = E_{5,1} -E_{7,3},	&\quad  Z_4  &= [X, Z_3 ] \ = E_{7,2} -E_{6,1}.
\end{alignat*}
The grading and full structure equations for $\lieg_{-}$ are therefore
\begin{gather*}
\begin{aligned}
\lieg_{-1} &= \langle\,  P_1,\, P_2, \, X\, \rangle, \\
\lieg_{-2} &= \langle\, \,  Y_1, \,Y_2\, \rangle, \\
\lieg_{-3} &= \langle \,Z_1, \, Z_2 \, \rangle, \\
\lieg_{-4} &= \langle \,Z_3 \, \rangle, \\
\lieg_{-5} &= \langle \,Z_4\, \rangle \\
\end{aligned}
\quad\text{and} 
\quad
\begin{tabular}{C|CCCCCCCCCC}
	& P _1 &P_2 &  X  & Y_1 & Y_2 & Z_1 &Z_2 & Z_3 & Z_4 \\
\hline 
P_1 &  0 & 0  & Y_1  & 0      & Z_1    & 0     & Z_3 & 0 & 0  \\
P_2 &     & 0  & Y_2  & Z_1  & Z_2    & Z_3 &    0  & 0 & 0  \\
X     &     & &0 & 0&0 &0 &0&Z_4 &0 \\
Y_1 &     & & &0 & 0&0 &-Z_4 &0  &0 \\
Y_2 &     & & & & 0&-Z_4 &0  &0 &0 \\
Z_1 &&&&&& 0 &0 &0&0 \\
Z_2 &&&&&&&0 &0 &0\\
Z_3 &&&&&&& &0 &0 \\
Z_4 &&&&&&&&& 0 \\
\end{tabular} .
\end{gather*}
	In terms of the dual basis  
	$\{\, \theta_{p}^1,\,  \theta_{p}^2, \, \theta_x,\,  \theta_{y}^1,\,\theta_{y}^2,\, \theta_{z}^1,\,   \theta_{z}^2,\,   \theta_{z}^3,\,  \theta_{z}^4\}$  the structure equations for  $\lieg_{-}$ are
\begin{gather*}
	d\theta^1_p = 0, \quad   d\theta^2_p = 0, \quad d\theta_{x} = 0,\quad  d\theta_y^1 = \theta_{x} \wedge \theta_{p}^1, 
	\quad   d\theta_y^2 = \theta_{x} \wedge \theta_{p}^2,
\\
	d\theta^1_z  = \theta_y^2  \wedge \theta_p^1 + \theta_y^1 \wedge \theta_p^2,
	\quad d\theta^2_z = \theta_y^2  \wedge \theta_p^2, \quad
	d\theta^3_z =  \theta^1_z\wedge \theta^2_p + \theta^2_z\wedge \theta^1_p, 
\\
d\theta^4_z =   \theta^1_y  \wedge  \theta^2_z +    \theta^2_y  \wedge  \theta^1_z  + \theta^3_z \wedge \theta_x. 
\end{gather*}
	Integrating these equations, 
	 one finds that
\begin{gather*}
	\theta^1_p = dp^1, \quad \theta^2_p= dp^2\quad \theta_{x}  = dx, \quad  \theta_y^1 =dy^1 - p^1 dx, \quad \theta_y^2= dy^2 - p^2 dx, 
\\
	\theta^1_z  =  dz^1 - p^2dy^1 - p^1dy^2 + p^1p^2dx, \quad	 \theta^2_z  = dz^2 -p^2dy^2 +\frac12(p^2)^2dx, \quad 
\\
	\theta^3_z  = dz^3 + \frac12(p^2)^2 dy^1 + p^1p^2 dy^2 -p^2dz^1 - p^1dz^2 - \frac12 p^1(p^2)^2 dx, \quad 
\\
\theta^4_z  =  dz^4  +  y^2dz^1 + y^1dz^2 -  x\, dz^3. \quad	
\end{gather*}
	The standard Pfaffian differential system  for  the parabolic geometry $B_3\{\alpha_1, \alpha_2, \alpha_3\}$  is therefore
\begin{align*}
	I_{{B_3\{1,2,3\}} }
&	=  \{\theta_y^1,\, \theta_y^2,\, \theta^1_z, \, \theta^2_z,\, \theta^3_z,\, \theta^4_z\,\}
\\
&	=  \{dy^1 - p^1 dx,\, dy^2 - p^2 dx, dz ^1 - p^1p^2 dx,\,  dz^2 - \frac12(p^2)^2 dx,\,  dz^3 - \frac12 p^1(p^2)^2 dx,\, 
\\
&	\quad \quad   dz^4	- \frac12p^2(xp^1p^2 - y^1p^2 - 2y^2 p^1)\, dx \}
\end{align*}
	which is the canonical Pfaffian system for the first order Monge equations \eqref{excpB2}.
	Given the visual  asymmetry of these equations,
	it is a remarkable fact that the symmetry algebra  is isomorphic to $\lies\lieo(4,3)$.
\par
\medskip
\noindent
{\bf Va. $\boldsymbol{G_2\{ \alpha_1 \}$ }. }
	Let $\{H_1, H_2\}$  be a Cartan subalgebra for $\lieg_2$ and  let  $Y_1, Y_2, Y_3, Y_4, Y_5, Y_6$ be bases 
	for the root spaces  for  the negative roots $-\alpha_1$, $-\alpha_2$, $-\alpha_1 - \alpha_2$, $-2\alpha_1 - \alpha_2$, 
	$-3\alpha_1 - \alpha_2$, $-3\alpha_1 - 2 \alpha_2$.  In terms of a Chevalley basis (see \cite[p. 346]{fulton-harris:1991a}), 
	the structure equations  for $\lieg_2$ are,  in part,
\begin{gather*}
\begin{tabular}{C|CCCCCCCCC}
	& H _1 &H_2 &  Y_1 & Y_2 & Y_3 & Y_4 &Y_5 & Y_6 \\
\hline 
H_1 &  0 & 0  & -2Y_1  & 3Y_2  & Y_3    & -Y_4     &-3 Y_5 &0   \\
H_2 &     & 0  & Y_1  & -2Y_2  & -Y_3   & 0 &    Y_5  & -Y_6   \\
Y_1     &     &  & 0& -Y_3 & -2Y_4 & 3Y_5&0 &0 \\
Y_2 &     & & &0 & 0 & 0 & Y_6  &0 \\
Y_3 &     & & & &0& 3Y_6&0&0  \\
Y_4 &&&&&& 0 &0 &0\\
Y_5 &&&&&&&0 &0\\
Y_6 &&&&&&& &0 \\
\end{tabular} .
\end{gather*}
	For $\Sigma =  \{\,\alpha_1 \}$ the roots of  height 1 are $\alpha_1$ and $\alpha_1 + \alpha_2$ and  thus $\lieg_{-}$ is 
	spaned  by the vectors 
\begin{equation*}
	Q= Y_3, \ X = Y_1,  \ P =  [Q, X] = 2Y_4,\ Y = [P, X] =  -6Y_5, \ Z = [Q,P] = 6Y_6.
\end{equation*}
	The structure equations for the dual coframe  $\{\, \theta^q, \theta^x, \theta^p,\theta^y, \theta^z\}$ are
\begin{equation*}
	d \theta^q = 0, \quad d \theta^x = 0, \quad d\theta^p = \theta^x \wedge \theta^q, \quad d\theta^y = \theta^x \wedge \theta^p, \quad 
	d \theta^z = \theta^p \wedge \theta^q,
\end{equation*}
	which are easily integrated to give 
\begin{gather*}
	\theta^q = dq, \quad   \theta^x = dx, \quad  \theta^p =  dp - q\,dx, \quad \theta^y = dy - p\,dx,
\\
	\theta^z =  dz - q\,dp  + \frac12 q^2 \,dx.
\end{gather*}
	The standard differential system for $\lieg_2\{\alpha_1\}$ is therefore 
\begin{equation*}
	I_{G_2\{1\}} = \text{span} \,\{ \theta^y, \theta^p, \theta^z \}  = \text{span }\,\{ dy - p\,dx, \, dp - q\,dx, dz  - \frac12 q^2 \,dx \} ,
\end{equation*} 
	which is the canonical Pfaffian system for the Cartan-Hilbert equation \EqRef{StandardG}.
\par
\medskip
\noindent
{\bf Vb. $\boldsymbol{G_2\{ \alpha_1,\, \alpha_2\, \}$ }.  }
	In this case  the roots of  height 1 are  $ \{ \alpha_1,\, \alpha_2\, \}$  so that $\lieg_{-}$ is the sum of all the negative root spaces. We 
set
\begin{gather*}
	R = Y_2, \quad  X = Y_1,  \quad   Q =  [R, X] = Y_3,\quad P   =  [Q, X] =  2Y_4, \     
\\
	Y =  [P ,X] =  -6Y_5,  \quad           Z = [Y , R ] = 6Y_6.
\end{gather*}	
	The structure equations for the dual coframe  $\{\, \theta^r, \theta^x,  \theta^q, \theta^p,\theta^y, \theta^z\, \}$ are
\begin{gather*}
	d \theta^r = 0, \quad d \theta^x = 0, \quad d\theta^q = \theta^x \wedge \theta^r, \quad d\theta^p = \theta^x \wedge \theta^q,
\\ 
	d \theta^y = \theta^x \wedge \theta^p,    \quad d \theta^z = \theta^r \wedge \theta^y + \theta^p \wedge \theta^q
\end{gather*}
	which are easily integrated to give 
\begin{gather*}
	\theta^r = dr, \quad   \theta^x = dx, \quad  \theta^q =  dq - r\,dx, \quad \theta^p = dp - q\,dx,
\\
	\theta^y = dy - p\, dx, \quad \theta^z =  dz   + r\,dy   - q\,dp   + (\frac12 q^2 - pr) \,dx.
\end{gather*}
	The standard differential system for $\lieg_2\{\alpha_1\}$ is therefore 
\begin{equation*}
	I _{G_2\{1,2\}}= \text{span} \,\{ \theta^y, \theta^p, \theta^q, \theta^z \}  = \text{span }\,\{ dy - p\,dx, \, dp - q\,dx,  dq - r\, dx, dz  - \frac12 q^2 \,dx \} ,
\end{equation*} 
	which is the canonical Pfaffian system for the  prolongation of the  Pfaffian system for the  Cartan-Hilbert equation \EqRef{StandardG} given in the previous case.

\section{Infinitesimal Symmetries for the Standard Models}
	In this section we give explicit formulas  for the infinitesimal symmetries for the Monge equations
	in Theorem A. We find that these infinitesimal symmetries are all prolonged  point symmetries and that
	coefficients of the vector fields for any symmetry
	are all quadratic functions of the  variables $x$, $y^i$, $z^\alpha$. 

	The infinitesimal symmetries for any first order system of Monge equations 
\begin{equation*}
	\dot z ^\alpha = F^\alpha(x, y^i,\dot y^i, z^\alpha)
\end{equation*}
	is, by definition, the Lie algebra of  vector fields
\begin{equation}
	X = A \frac{\partial\hfill }{\partial x}  +  B^i \frac{\partial\hfill }{\partial y^i} + 
	C^\alpha \frac{\partial\hfill }{\partial z^\alpha} + D^i \frac{\partial\hfill }{\partial \dot y^i}, 
\EqTag{SymX}
\end{equation}	
	where the coefficients $A$, $B^i$, $C^\alpha$, $D^i$ are functions of the variables 
	$x$, $y^i$, $z^\alpha$, $\dot y^i$, which preserve the Pfaffian system
\begin{equation*}
	\mathcal I  = \text{span}\{\,\theta^i = dy^i - \dot y^i dx,  
	\theta^\alpha = dz^\alpha - F^\alpha \,dx\,\}. 
\end{equation*}
	From the equation $\CalL_X \theta^i \equiv 0 \mod \mathcal I$ one finds that the coefficients $A$ and 
	$B^i$ are independent of the variables $\dot y^i$  and that 
\begin{equation}
	D^i = D_x B^i - \dot y^i D_x A, \quad\text{where}\quad 
	D_x =   \frac{\partial\hfill }{\partial x} +  \dot y^i\frac{\partial\hfill }{\partial y^i}  +
	F^\alpha\, \frac{\partial\hfill }{\partial z^\alpha}. 
\EqTag{Di}
\end{equation}
	The equation $\CalL_X \theta^\alpha \equiv 0 \mod \mathcal I$ then implies {\bf[i]} that the coefficients 
	$C^\alpha$ are also independent of the variables $\dot y^i$ so that $X$ is a prolonged 
	point transformation,  and {\bf [ii]}
\begin{equation}
	D_x C^\alpha - X(F^\alpha) - F^\alpha D_x(A)  =0. 
\EqTag{SymEq}
\end{equation}
	
	To continue, we now take $F^\alpha  = F^\alpha_{ij} \dot y^i \dot y^j$,
	where the coefficients  $F^\alpha_{ij} = F^\alpha_{ji}$ are constant. 
	Then, by equation \EqRef{SymX}, we find that \EqRef{SymEq} becomes
\begin{equation*}
	D_x C^\alpha - 2F^\alpha_{ij} D_xB^i \dot y^j + F^\alpha D_x(A)  =0. 
\end{equation*}
	This equation is a  polynomial identity in the derivatives  $\dot y^j$ of order 4. From the 
	coefficients  of $\dot y^i\dot y^j\dot y^h\dot y^k$ and $1$ one finds that 
\begin{equation}
	\frac{\partial A }{\partial z^\alpha} = 0 \quad\text{and}\quad  \frac{\partial C^\alpha }{\partial x} = 0.
\EqTag{NoZNoX}
\end{equation}
	The coefficients of $\dot y^i$, $\dot y^i\dot y^j$ and $\dot y^i\dot y^j \dot y^k$ give, respectively,
\begin{subequations}
\EqTag{sym2}
\begin{align}
	2 F^\alpha_{\ell i} \frac{\partial B^\ell}{\partial x} &=  \frac{\partial C^\alpha }{\partial y^i},
\EqTag{sym2a}
\\[2\jot]
	F^\alpha_{\ell i} \frac{\partial B^\ell}{\partial y^j} 
	+ F^\alpha_{\ell j} \frac{\partial B^\ell}{\partial y^i}
	&= F^\alpha_{ij} \frac{\partial A }{\partial x} +  
	F^\beta _{ij} \frac{\partial C^\alpha }{\partial z^\beta}, \quad \text{and}
\EqTag{sym2b}
\\[2\jot]
	2F^\alpha_{\ell (i}F^\beta_{j k)}\frac{\partial B^\ell }{\partial z^\beta} &= 
	F^\alpha_{(ij}\frac{\partial A }{\partial y^{k)} }.
\EqTag{sym2c}
\end{align}
\end{subequations}
	These are the determining equations for the symmetries of the Monge equations 
	$\dot z^\alpha =  F^\alpha_{ij} \dot y^i \dot y^j$. 
	
		The integrability conditions for \EqRef{NoZNoX} and 
\EqRef{sym2} imply that all the coefficients $A, B^i$ and $C^\alpha$  are quadratic functions of the coordinates
	$\{x, y^i, z^\alpha\}$.
Thus the determining equations reduce to purely algebraic equations. It is now a straightforward, albeit a slightly tedious,
	matter to explicitly  construct a basis for all the solutions to the determining equations. The results of 
	these calculations are summarized in the following table of symmetries for the Monge equations of 
	type $A$, $BD$ and $C$.

\bigskip
\hbox{
\kern - 40pt
\begin{tabular}{|l|l|l|l|l|l|l|}
\hline
gr. &\rule[-5pt]{0pt}{19pt}$\dot z^i = \dot y^0 \dot y^i$  &   $\dot z = \frac12 \kappa_{ij}\dot y^i\dot y^j$   & $\dot z^{ij} = \dot y^i \dot y^j$	
\\
\hline
$-1$ 
&\rule[-5pt]{0pt}{19pt}$\partial_x$, $\partial_{y^0}$, $\partial_{y^i}$, $\partial_{z^i}$   
& $\partial_x$,  $\partial_{y^i}$, $\partial_{z}$
& $\partial_x$,  $\partial_{y^i}$, $\partial_{z^{ij}}$
\\
\hline
\rule[-7pt]{0pt}{21pt}
0 
& $x\,\partial_x +   \frac12 y^0\,\partial_{y^0} +\frac12 y^i\,\partial_{y^i}$
&  $x\,\partial_x + \frac12 y^i\,\partial_{y^i}$
& $x\,\partial_x  + \frac12 y^{j}\, \partial_{y^j}$
\\
\cline{2-4}
&\parbox[t][23.5pt][b]{60pt}{\begin{align*} &y^0\,\partial_x + z^i\,\partial_{y^i}\, ,
\\[-1\jot] & y^i\, \partial_x  +z^i\,\partial_{y^0} 
\\[-8.5\jot] \end{align*} }
& $y^i\, \partial_x + \frac12 z\kappa ^{ij}\ \partial_{y^j}$
& $y^i\, \partial_{x} + \frac12 z^{ij}\, \partial_{y^{j}}$
\\
\cline{2-4}
&\parbox[t][23.5pt][b]{60pt}{\begin{align*} &x\partial_{y^0} + y^i\partial_x\, , 
\\[-1\jot] & x\partial_{y^i} + y^0 \partial_{z^i} 
\\[-8.5\jot] \end{align*} }
& $x\, \kappa^{ij}\partial_y^{j} + y^i\, \partial_z$
&\parbox[t][23.5pt][b]{60pt}{\begin{align*} &x\, \partial_{y^{i}} + 2y^j\, \partial_{{ij}}
\\[-1\jot] &\partial_{ij} = \tfrac{1+\delta_{ij}}{2} \partial_{z^{ij}}
\\[-8.5\jot] \end{align*} }
\\
\cline{2-4}
&\parbox[t][23.5pt][b]{60pt}{\begin{align*} &y^0\partial_{y^0}  + z^i \partial_{z^i}\, ,
\\[-1\jot] & y^i\partial_{y^j} + z^i\partial_{z^j}
\\[-8.5\jot] \end{align*} }
&\parbox[t][23.5pt][b]{60pt}{\begin{align*} &y^i\, \partial_{y^i} + 2z\,\partial_z\, ,
\\[-1\jot] & \kappa^{ik}b_{ij}y^j\,\partial_{y^k}\ (b_{ij}\text{ skew})
\\[-8.5\jot] \end{align*} }
& $y^i\, \partial_{y^j} + 2 z^{ik}\, \partial_{{jk}} $
\\
\hline
1
&\parbox[t][23.5pt][b]{80pt}{\begin{align*} &x^2\partial_x +  xy^0\partial_{y^0} + xy^i\partial_{y^i}
\\[-1\jot] &  + y^0y^i\partial_{z^i}
\\[-8.5\jot] \end{align*} }
& \parbox[t][23.5pt][b]{80pt}{\begin{align*} & x^2 \partial_x  +   zy^i \partial_{y^i} +  K\partial_z
\\[-1\jot] & K = \kappa_{ij} y^i y^j
\\[-8.5\jot] \end{align*} }
&\parbox[t][23.5pt][b]{60pt}{\begin{align*} &x^2 \partial_x + xy^i\partial_{y^i}
\\[-1\jot] &  + y^iy^j \partial_{z^{ij}}
\\[-8.5\jot] \end{align*} }
\\
\cline{2-4}
&
\parbox[t][37pt][b]{80pt}{\begin{align*}& xy^0 \partial_x + (y^0)^2 \partial_{y^0} + y^0z^i \partial_{z^i}\, ,
\\[-1\jot] & xy^i\partial_x + xz^i \partial_{y^0}  + y^iy^j\partial_{y^j} 
\\[-1\jot] & + z^iy^j\partial_{z^j}
\\[-9.5\jot] \end{align*} }

&  \parbox[t][23pt][b]{80pt}{\begin{align*} &xy^i  \partial_x +  (y^i y^j  + (xz- K) \kappa^{ij})\partial_{y^j} 
\\[-1\jot] & + zy^i\partial_z
\\[-9.5\jot] \end{align*} }
& \parbox[t][23pt][b]{80pt}{\begin{align*} &xy^i  \partial_x +  y^jz^{ik}\partial_{{jk}}
\\[-1\jot] & + {\textstyle \frac12}(xz^{ij} + y^iy^j)\partial_{y^i}
\\[-9\jot] \end{align*} } 
\\
\cline{2-4}
&\parbox[t][23pt][b]{80pt}{\begin{align*} &y^0y^i \partial_x + y^0z^i\partial_{y^0} + y^iz^j\partial_{y^j} 
\\[-1\jot] &  + z^i z^k\partial_{z^k}
\\[-9.5\jot] \end{align*} }
&  $K \partial_x   + zy^i\partial_{y^i} + z^2 \partial_z$
&\parbox[t][23pt][b]{70pt}{\begin{align*}&y^i y^j \partial_x + z^{ih}z^{jk} \partial_{{hk}}
\\[-1\jot] &+ {\textstyle \frac12}(y^iz^{jk} + y^jz^{ik})\partial_{y^k} 
\\[-9\jot] \end{align*} }
\\
\hline
\end{tabular}
}

\bigskip

It is not difficult to see that {\it the above point symmetries for our Monge equations coincide with the infinitesimal generators for the $G$ action on $G/\tilde P$, where $\tilde P$ is the 
parabolic subgroup of $G$ defined by the $|1|$-gradings using the leader only.}


\end{document}